\documentclass[hidelinks,onefignum,onetabnum]{siamart251216}

\usepackage{amsfonts}
\usepackage{graphicx}
\usepackage{epstopdf}
\usepackage{algorithmic}
\usepackage{placeins}
\ifpdf
  \DeclareGraphicsExtensions{.eps,.pdf,.png,.jpg}
\else
  \DeclareGraphicsExtensions{.eps}
\fi

\newsiamremark{remark}{Remark}
\newsiamremark{hypothesis}{Hypothesis}
\crefname{hypothesis}{Hypothesis}{Hypotheses}
\newsiamthm{claim}{Claim}
\newsiamremark{fact}{Fact}
\crefname{fact}{Fact}{Facts}

\headers{Mean Field Games without Perfect Recall}%
{Xuanping Zhang, Xiao Zhang, and Wang Yao}

\title{Partially Observed Mean Field Games without Perfect Recall: Optimality Conditions and Equilibria\thanks{Submitted to the editors September 1, 2026. This work has been submitted to the SIAM for possible publication. Copyright may be transferred without notice, after which this version may no longer be accessible. \funding{This work was supported by the National Natural Science Foundation of China (NSFC Grant No.12441101 and 12601852), the Research Funding of Hangzhou International Innovation Institute of Beihang University (Grant No.2024KQ161), and Beijing Advanced Innovation Center for Future Blockchain and Privacy Computing.}}}

\author{Xuanping Zhang\thanks{School of Mathematical Sciences and LMIB, Beihang University, Beijing 100191, China; Hangzhou International Innovation Institute of Beihang University, Hangzhou 311115, China (\email{BY2409034@buaa.edu.cn}).} \and Xiao Zhang\thanks{School of Mathematical Sciences and LMIB, Beihang University, Beijing 100191, China (\email{xiao.zh@buaa.edu.cn}).} \and Wang Yao\thanks{Corresponding author. School of Artificial Intelligence, LMIB, and Beijing Advanced Innovation Center for Future Blockchain and Privacy Computing, Beihang University, Beijing 100191, China (\email{yaowang@buaa.edu.cn}).}}

\usepackage{amsopn}

\ifpdf
\hypersetup{
  pdftitle={Partially Observed Mean Field Games without Perfect Recall: Optimality Conditions and Equilibria},
  pdfauthor={Xuanping Zhang, Xiao Zhang, and Wang Yao}
}
\fi

\begin{document}

\maketitle

\begin{abstract}
This paper studies partially observed mean field games without perfect recall (WPR). The representative agent observes a noisy signal, but the control at time \(t\) uses only \(\mathcal G_t^I=\sigma(y_t)\), a generally non-nested information family.  The conditional population law instead uses the observation filtration \(\mathbb F^Y\). These coupled levels rely on different information scales and are difficult to close within one construction.  We parameterize the environment by a deterministic compatible joint law of state, driving variables, and random mean field term, thereby preserving its dependence structure without enlarging the agent's control information. For a fixed law, a reference measure and Girsanov's theorem yield a WPR stochastic maximum principle; the selected response is represented by the conditional Hamiltonian and WPR belief measure.  The joint path posterior of hidden state and mean field term gives a weak Kushner-Stratonovich representation of the conditional population law. A recursive response map is continuous on a compact convex set of compatible laws, so Schauder-Tychonoff yields a weak WPR equilibrium.  For a fixed equilibrium law and feedback, a compatible Yamada-Watanabe theorem lifts pathwise uniqueness to a strong realization.  Finally, a linear-quadratic interbank lending example compares perfect recall (PR) with WPR.  The PR response follows the Kalman-Bucy feedback, whereas the WPR response solves a Fredholm-Volterra equation and is affine in the current observation under Gaussianity.  The numerical experiment illustrates how equilibrium behavior differs between PR and WPR.
\end{abstract}



\section{Introduction}

Mean field game (MFG) theory studies stochastic differential games with many weakly interacting agents.  Initiated by Lasry and Lions \cite{lasry2007mean} and Huang, Malham\'{e}, and Caines \cite{huang2006large}, it has both analytic and probabilistic formulations.  A standard construction fixes a candidate measure flow, solves the representative agent's control problem, and matches that flow with the law generated by the optimal response.  General accounts are given by Bensoussan, Frehse, and Yam \cite{bensoussan2013mean} and Carmona and Delarue \cite{carmona2018prob1,carmona2018prob2}.  When the agent observes only a noisy signal, the information structure also becomes part of the model.

Partially observed MFGs have been studied under several information and noise structures.  Caines and Kizilkale \cite{caines2016epsilon} analyze partially observed major-minor LQG equilibria; Huang, Wang, and Wu \cite{huang2016backward} compare backward LQG games under full and partial information; and Bensoussan, Feng, and Huang \cite{bensoussan2021lqgpartial} treat common noise through conditional estimates. Sen and Caines \cite{sen2016mean} formulate a nonlinear major-minor model with consistency in a Wasserstein space of random probability measures; Sen and Caines \cite{sen2019partial} derive partially observed best responses through nonlinear filtering. Carmona, Delarue, and Lacker \cite{carmona2016common} construct MFGs with common noise through the joint law of the initial state, noises, random measure, relaxed control, and state, using compatibility to preserve nonanticipativity.  At the individual control level, Bensoussan \cite{bensoussan1992stochastic} and Bandini, Cosso, Fuhrman, and Pham \cite{bandini2019randomized} formulate filtering problems through conditional laws.  These law-based formulations retain the hidden state or random mean field term together with its dependence on randomness.

These studies retain perfect recall (PR): admissible controls and the conditional population law share the increasing observation filtration \(\mathcal F_t^Y\).  We consider instead the case without perfect recall (WPR), where the control uses only
\[
    \mathcal G_t^I=\sigma(y_t).
\]
This family is generally non-nested.  The separation principle and dynamic programming therefore do not apply directly; optimality conditions must be projected onto the information available at each time.  Inspired by Witsenhausen's counterexample \cite{witsenhausen1968counterexample}, subsequent work has developed systematic approaches to stochastic control under nonclassical information structures.  In continuous time, Charalambous and Ahmed \cite{charalambous2016centralized} derive decentralized optimality conditions under history-based local filtrations, which retain perfect recall.  More recently, Charalambous \cite{charalambous2025decentralized} treats non-cooperative stochastic differential games whose information structures may lack perfect recall.

Bringing WPR into an MFG adds mean field consistency to this nonclassical control problem.  The individual response uses the current observation, whereas the conditional population law \(m_t^Y=\mathcal L(x_t^*\mid\mathcal F_t^Y)\) depends on the full observation history.  Under PR, one filtration supports both objects.  Under WPR, the full history filter contains information unavailable to the controller, while the current observation lacks the temporal information needed to update the random mean field term.  That mean field term remains coupled to the driving variables, so treating it as deterministic or observable would lose that coupling or enlarge the admissible information.  Changing the response also changes the state law and its observation-dependent population law; time marginals alone do not retain the coupling with signals and noises.  The difficulty is to close a current-information best response and a history-dependent consistency condition within one construction.

We develop a law-based formulation tailored to this structure.  A deterministic compatible joint law \(P\) of \((X,V,\mu)\) parameterizes the random environment and retains its dependence structure without revealing the mean field term to the controller.  For fixed \(P\), an instantaneous WPR belief determines the representative control, while an observation-history path posterior determines the conditional population law.  A response map on compatible joint laws closes the two levels and yields weak equilibria with strong compatible realizations.  An LQ interbank lending model then permits a direct PR--WPR comparison.

The main contributions are as follows.  First, we formulate the representative agent's best-response problem in a random environment indexed by a compatible joint law \(P\).  Under a reference probability measure, we construct its variational and adjoint systems, derive a WPR stochastic maximum principle, and represent the conditional Hamiltonian through the WPR belief.  Convexity and belief regularity yield a selected optimal response with a Lipschitz feedback representation \(u_t^*=\phi^P(t,y_t)\), without enlarging the WPR information available to the agent.

Second, we construct an optional joint path posterior for the hidden state and the random mean field term under the observation filtration \(\mathbb F^Y\).  The full posterior retains the histories required by path-dependent coefficients, and its current state marginal is the conditional population law.  We identify the innovation Brownian motion and the predictable representation property, and derive a weak Kushner-Stratonovich representation for \(m_t^Y\) whose coefficients are evaluated under the path posterior.  Although the equation need not close in \(m_t^Y\) alone, it uses the full history law without enlarging WPR information and keeps the population filter distinct from the instantaneous WPR belief.

Third, we construct a joint law response map \(\mathcal R\) through a recursive conditional-law scheme.  Uniform moment, time-regularity, and stability estimates make it a continuous self-map of a compact convex set.  The Schauder-Tychonoff theorem then yields a weak WPR-MFG equilibrium, whose fixed point is identified with the optimal controlled law.  For the stated response class, pathwise uniqueness holds for a fixed equilibrium law and feedback; the compatible Yamada-Watanabe theorem then yields a strong realization of the equilibrium.  We also analyze an LQ interbank lending model based on Carmona, Fouque, and Sun \cite{carmona2015systemic}. It yields Kalman-Bucy feedback under PR and a Fredholm-Volterra equation under WPR; the Gaussian WPR solution is affine in the current observation.  The numerical experiment complements this analysis by comparing the equilibria that arise under the two information structures.

Section 2 presents the model and reference measure formulation; Sections 3--5 develop the fixed-\(P\) response, conditional mean field law, and equilibrium construction; Section 6 gives the LQ example and numerics.  Proofs are collected in the appendices.

\section{Problem Statement}

Fix a finite horizon $[0,T]$.  Let $(\Omega^0,\mathcal F^0,\mathbb P^0)$ support independent standard Brownian motions $W$ and $B$, with values in $\mathbb R^{d_w}$ and $\mathbb R^{d_b}$, respectively.  Let $\mathbb F^0$ be their natural filtration with the usual $\mathbb P^0$-augmentation.  For a complete separable metric space $(E,d_E)$, let $\mathcal P(E)$ be the Borel probability measures on $E$ with the weak topology, and set
\begin{equation*}
    \mathcal{P}_2(E) := \left\{ \mu \in \mathcal{P}(E) : \int_E d_E(x, x_0)^2 \mu(dx) < \infty \text{ for some } x_0 \in E \right\}.
\end{equation*}
For $\mu,\nu\in\mathcal P_2(E)$, define
\begin{equation}
    \mathcal{W}_2(\mu, \nu) := \left( \inf_{\pi \in \Pi(\mu, \nu)} \int_{E \times E} d_E(x, y)^2 \pi(dx, dy) \right)^{\frac{1}{2}},
\end{equation}
where $\Pi(\mu,\nu)$ is the set of couplings.  For Euclidean $E$, write $C([0,t];E)$ for the continuous paths with norm $\|x\|_t:=\sup_{0\le s\le t}|x_s|$.

\subsection{System Dynamics and WPR Information Structure}

Fix an exogenous flow $\boldsymbol\mu\in C([0,T];\mathcal P_2(\mathbb R^{d_x}))$.  The representative state satisfies
\begin{equation} \label{eq:true_state}
    dx_t = b(t, x, y, u_t, \mu_t) dt + \sigma_1(t, x, y) dW_t, \quad x_0 \in \mathbb{R}^{d_x},
\end{equation}
where $x=x_{\cdot\wedge t}$ and $y=y_{\cdot\wedge t}$ are the state and observation paths.  The control takes values in $\mathbb U\subset\mathbb R^{d_u}$, and $\mu_t$ is the prescribed population law.

The state is observed only through
\begin{equation} \label{eq:obs_dyn}
    dy_t = h(t, x, y, u_t, \mu_t) dt + \sigma_2(t, y) dB_t, \quad y_0 = 0,
\end{equation}
where $B$ is independent of $W$.  Set \(X_t:=\begin{bmatrix}x_t^\top&y_t^\top\end{bmatrix}^\top\) and \(V_t:=\begin{bmatrix}W_t^\top&B_t^\top\end{bmatrix}^\top\).  Then \eqref{eq:true_state}--\eqref{eq:obs_dyn} become

\begin{equation} \label{eq:augmented_dyn}
    dX_t = F(t, X, u_t, \mu_t) dt + \Sigma(t, X) dV_t,
    \qquad
    X_0 = \begin{bmatrix} x_0 \\ 0 \end{bmatrix},
\end{equation}
where
\begin{equation} \label{eq:augmented_coefficients}
    F(t, X, u_t, \mu_t) := \begin{bmatrix} b(t, x, y, u_t, \mu_t) \\ h(t, x, y, u_t, \mu_t) \end{bmatrix}, \quad
    \Sigma(t, X) := \begin{bmatrix} \sigma_1(t, x, y) & 0 \\ 0 & \sigma_2(t, y) \end{bmatrix}.
\end{equation}

At time $t$, the agent uses only the current observation:
\begin{equation} \label{eq:info_structure}
    \mathcal G_t^I=\sigma(y_t).
\end{equation}

\begin{remark} \label{remark_WPRInfo}
The family $\mathbb G^I:=\{\mathcal G_t^I\}_{0\le t\le T}$ is generally not a filtration because $\sigma(y_s)\subseteq\sigma(y_t)$ may fail for $s<t$.  We call it a non-nested instantaneous information family; this is the meaning of WPR used here.
\end{remark}

Let $(\mathbb U,\mathcal B(\mathbb U))$ be the action space and let $\Phi_{ad}$ be the strategy class specified in Assumption (H5).  Define {\small
\begin{equation} \label{eq:admissible_control}
    \mathcal U_{ad}:=\left\{u\,\middle|\,
    u_t=\phi(t,y_t)\in\mathbb U\ dt\otimes d\mathbb P\text{-a.e.},\
    \phi\in\Phi_{ad},\
    \mathbb E^{\mathbb P^u}\!\left[\int_0^T |u_t|^2dt\right]<\infty
    \right\}.
\end{equation}
}
For the fixed flow $\mu$, the representative agent minimizes
\begin{equation} \label{eq:cost_func}
    J^{\mathbb{P}^u}(u) = \mathbb{E}^{\mathbb{P}^u} \left[ \int_0^T l(t, X, u_t, \mu_t) dt + g(X_T) \right],
\end{equation}
over $\mathcal U_{ad}$; the running cost may depend on $X_{\cdot\wedge t}$.

\subsection{Equivalent Stochastic System Formulation via Reference Measure}

We fix the uncontrolled path law under a reference measure and introduce the control through a likelihood process.  Let $(\Omega^0,\mathcal F^0,\mathbb F^0,\mathbb P)$ support a standard Brownian motion $V$, and set

\begin{equation} \label{eq:ref_dyn}
    dX_t = \Sigma(t, X) dV_t, \quad X_0 = \begin{bmatrix} x_0 \\ 0 \end{bmatrix}.
\end{equation}
The law of $X$ under $\mathbb P$ is independent of the admissible control.

For $u\in\mathcal U_{ad}$ define
\begin{equation} \label{eq:radon_nikodym}
    \Lambda^u(t) = \exp \left( \int_0^t \theta^\top(s) dV_s - \frac{1}{2} \int_0^t \|\theta(s)\|^2 ds \right),
\end{equation}
where $\theta(t) := \Sigma^{-1}(t, X) F(t, X, u_t, \mu_t)$.  Then
\begin{equation} \label{eq:lambda_dyn}
    d\Lambda^u(t) = \Lambda^u(t) \theta^\top(t) dV_t, \quad \Lambda^u(0) = 1.
\end{equation}

The controlled probability measure associated with $u$ is defined by
\begin{equation} \label{eq:measure_change}
    \left. \frac{d\mathbb{P}^u}{d\mathbb{P}} \right|_{\mathcal{F}_t^0} = \Lambda^u(t), \quad \forall t \in [0,T].
\end{equation}

\begin{remark} \label{remark_RefMeasureMartingale}
The measure change is used only after the standing assumptions establish that $\Lambda^u$ is a uniformly integrable martingale.
\end{remark}

If $\Lambda^u$ is a true martingale, Girsanov's theorem gives
\begin{equation} \label{eq:BM_Original}
    V_t^u := V_t - \int_0^t \Sigma^{-1}(s, X) F(s, X, u_s, \mu_s) ds
\end{equation}
as a Brownian motion under $\mathbb P^u$, and $X$ solves
\begin{equation}
    dX_t = F(t, X, u_t, \mu_t) dt + \Sigma(t, X) dV_t^u.
\end{equation}

The cost can also be written under $\mathbb P$.  Since $\Lambda^u(t)=\mathbb E^{\mathbb P}[\Lambda^u(T)\mid\mathcal F_t^0]$, Bayes' formula and Fubini's theorem give
\begin{align*}
\mathbb E^{\mathbb P^u}[g(X_T)]
&=\mathbb E^{\mathbb P}[\Lambda^u(T)g(X_T)],\\
\mathbb E^{\mathbb P^u}\!\left[\int_0^T l(t,X,u_t,\mu_t)dt\right]
&=\mathbb E^{\mathbb P}\!\left[
\int_0^T\Lambda^u(t)l(t,X,u_t,\mu_t)dt\right],
\end{align*}
where the standing moment bounds justify the integrability.  Therefore
\begin{equation} \label{eq:cost_func_ref}
    J^{\mathbb{P}}(u)
    =
    \mathbb{E}^{\mathbb{P}}\!\left[
    \int_0^T \Lambda^u(t)l(t,X,u_t,\mu_t)dt
    +\Lambda^u(T)g(X_T)
    \right].
\end{equation}

The fixed-flow analysis uses the following assumptions.  For a continuous path $X$, write $\|X\|_t:=\sup_{0\le s\le t}|X_s|$.

\textbf{Assumption (H0) [Smoothness]:}

$(\mathrm{i})$ The coefficient functionals $F,\Sigma,l,g$ are smooth in $t$ and jointly continuous in $(X,u_t,\mu_t)$.

$(\mathrm{ii})$ The running cost $l$ is continuously Fréchet differentiable in the path $X$ and admits a twice continuously differentiable extension in the control variable to an open neighborhood of $\mathbb U$.  Control derivatives refer to this extension.

$(\mathrm{iii})$ The derivatives $\nabla_Xl$, $\nabla_ul$, and $\nabla_{uu}^2l$ are jointly continuous.

\textbf{Assumption (H1) [Control Space, Initial State, and Reference Basis]:}

$(\mathrm{i})$ The admissible control space $\mathbb{U} \subseteq \mathbb{R}^{d_u}$ is nonempty, compact, and convex.

$(\mathrm{ii})$ For some $\alpha_0>0$, the initial state satisfies
\begin{equation}
    \mathbb E^{\mathbb P}\!\left[\exp\!\left(\alpha_0|X_0|^2\right)\right]<\infty .
\end{equation}

$(\mathrm{iii})$ The reference filtration $\mathbb F^0$ is the usual augmentation of $\mathcal F_t^0=\sigma(X_0,V_s:0\le s\le t)$, where $X_0$ is independent of $V$. It has the predictable representation property with respect to $V$: each square-integrable $(\mathbb F^0,\mathbb P)$-martingale $M$ has the unique form $M_t=M_0+\int_0^t Z_s^\top dV_s$, where $Z$ is predictable and $\mathbb E^{\mathbb P}[\int_0^T|Z_s|^2ds]<\infty$.

\textbf{Assumption (H2) [Diffusion Properties]:}

$(\mathrm{i})$ The augmented diffusion matrix $\Sigma(t,X)$ and its inverse are progressively measurable and uniformly bounded.  For some $L_\Sigma>0$, {\small
\begin{equation}
\|\Sigma(t,X)-\Sigma(t,X')\|
+\|\Sigma^{-1}(t,X)-\Sigma^{-1}(t,X')\|
\le L_\Sigma\|X-X'\|_t
\end{equation}
}
for all $t\in[0,T]$ and paths $X,X'$.

$(\mathrm{ii})$ The matrix $\Sigma(t,X)\Sigma(t,X)^\top$ is uniformly positive definite.

\textbf{Assumption (H3) [Affine Drift Structure]:}

$(\mathrm{i})$ For appropriately dimensioned $F_0,F_1$, the augmented drift is affine in $u_t$:
\begin{equation}
    F(t, X, u_t, \mu_t) = F_0(t, X, \mu_t) + F_1(t, X, \mu_t) u_t.
\end{equation}

$(\mathrm{ii})$ For some $C_{F_0},C_{F_1}>0$ and all $t\in[0,T]$, $u_t\in\mathbb U$,
\begin{equation}
    |F_0(t, X, \mu_t)| \le C_{F_0} \left[ 1 + \|X\|_t + \mathcal{W}_2(\mu_t, \delta_0) \right], \quad |F_1(t, X, \mu_t)| \le C_{F_1}.
\end{equation}

$(\mathrm{iii})$ For some $L_{F_0},L_{F_1}>0$, all paths $X,X'$, and $\mu_t,\mu_t'\in\mathcal P_2(\mathbb R^{d_x})$,
\begin{align}
    |F_0(t, X, \mu_t) - F_0(t, X', \mu'_t)| &\le L_{F_0} \left[ \|X - X'\|_t + \mathcal{W}_2(\mu_t, \mu'_t) \right], \\
    |F_1(t, X, \mu_t) - F_1(t, X', \mu'_t)| &\le L_{F_1} \left[ \|X - X'\|_t + \mathcal{W}_2(\mu_t, \mu'_t) \right].
\end{align}

\textbf{Assumption (H4) [Cost Functional Properties]:}

$(\mathrm{i})$ For some $C_l,C_g>0$, the running and terminal costs satisfy
\begin{align}
    |l(t, X, u_t, \mu_t)| &\le C_l \left[ 1 + \|X\|_t^2 + |u_t|^2 + \mathcal{W}_2(\mu_t, \delta_0)^2 \right], \\
    |g(X_T)| &\le C_g \left( 1 + |X_T|^2 \right).
\end{align}

$(\mathrm{ii})$ For some $L_l>0$, all paths $X,X'$, controls $u_t,u_t'$, and distributions $\mu_t,\mu_t'$,
\begin{equation}
\begin{split}
    |l(t, X,u_t,\mu_t) - l(t, X', u'_t,\mu'_t)| \le& L_{l} \big[ \|X - X'\|_t + |u_t-u'_t| + \mathcal{W}_2(\mu_t, \mu'_t) \big] \\
    &\cdot \big[ 1 + \|X\|_t + \|X'\|_t + |u_t| + |u'_t| \\
    &\quad + \mathcal{W}_2(\mu_t, \delta_0) + \mathcal{W}_2(\mu'_t, \delta_0) \big].
\end{split}
\end{equation}

$(\mathrm{iii})$ For some $\lambda>0$, the running cost is strongly convex in $u_t$: $\nabla_{uu}^2l\ge\lambda I_{d_u}$.

$(\mathrm{iv})$ There exists $L_{\nabla l}>0$ such that, for all $t\in[0,T]$, $u\in\mathbb U$, paths $X,X'$, and $\mu_t,\mu_t'\in\mathcal P_2(\mathbb R^{d_x})$,
\begin{equation}
    |\nabla_ul(t,X,u,\mu_t)-\nabla_ul(t,X',u,\mu_t')|
    \le L_{\nabla l}\big[\|X-X'\|_t+\mathcal W_2(\mu_t,\mu_t')\big].
    \label{eq:running_cost_gradient_stability}
\end{equation}

\textbf{Assumption (H5) [Compact WPR Strategy Class]:}

For some $L_\Phi>0$, the strategy class in \eqref{eq:admissible_control} is {\small
\begin{equation}
    \Phi_{ad}:=\big\{\phi\in C([0,T]\times\mathbb R^{d_y};\mathbb U):
    |\phi(t,y)-\phi(s,z)|\le L_\Phi(|t-s|+|y-z|)\big\}.
    \label{eq:Phi_ad}
\end{equation}
}
For the fixed flow, assume also that some $p_\Lambda>1$ satisfies
\begin{equation}
    \sup_{u\in\mathcal U_{ad}}
    \mathbb E^{\mathbb P}\!\left[(\Lambda^u(T))^{p_\Lambda}\right]<\infty .
    \label{eq:H5_likelihood_bound}
\end{equation}
Thus $\Phi_{ad}$ is nonempty, convex, locally-uniformly closed, and uniformly bounded and equicontinuous on compact subsets.

\begin{theorem}\label{thm:ref_solution}
Suppose Assumptions (H0), (H1), and (H2) hold.
\begin{enumerate}
    \item There exists a pathwise unique continuous strong solution $X$ to the uncontrolled SDE \eqref{eq:ref_dyn}, adapted to $\mathbb F^0$ and satisfying $\mathbb E^{\mathbb P}[\sup_{t\le T}|X_t|^2]<\infty$.
    \item For any $p \ge 1$, there exists a constant $C_p > 0$ such that
    \begin{equation}
        \mathbb{E}^{\mathbb{P}} \left[ \sup_{t \in [0,T]} |X_t|^{2p} \right]
        \le C_p\left(1+\mathbb E^{\mathbb P}[|X_0|^{2p}]\right)<\infty.
    \end{equation}
    \item There exists $\alpha_X>0$ such that
    \begin{equation}
        \mathbb E^{\mathbb P}\!\left[
        \exp\!\left(\alpha_X\|X\|_T^2\right)\right]<\infty .
        \label{eq:reference_exponential_moment}
    \end{equation}
\end{enumerate}
\end{theorem}
\begin{proof}
See Appendix \ref{app:proof_ref_solution}.
\end{proof}

\begin{theorem}\label{thm:girsanov}
Suppose Assumptions (H0)--(H4) hold.  Let $\mu\in C([0,T];\mathcal P_2(\mathbb R^{d_x}))$ be fixed.  Set $C_4(T):=T\|\Sigma^{-1}\|_\infty^2C_{F_0}^2$ and assume $C_4(T)<\alpha_X$, where $\alpha_X$ is given by \eqref{eq:reference_exponential_moment}.  Then, for every $u\in\mathcal U_{ad}$, the following assertions hold.

\begin{enumerate}
\item The Novikov condition is satisfied on $[0,T]$:
\begin{equation} \label{eq:novikov}
    \mathbb{E}^{\mathbb{P}} \left[ \exp \left( \frac{1}{2} \int_0^T \|\Sigma^{-1}(t, X) F(t, X, u_t, \mu_t)\|^2 dt \right) \right] < \infty.
\end{equation}
Consequently, $\Lambda^u$ is a uniformly integrable $(\mathbb F^0,\mathbb P)$-martingale.

\item Let $\mathbb P^u$ be the equivalent controlled probability measure defined by \eqref{eq:measure_change}.  Then $V^u$ in \eqref{eq:BM_Original} is a Brownian motion under $\mathbb P^u$, and $(X,V^u)$ is a weak solution of \eqref{eq:augmented_dyn}, unique in law.
\end{enumerate}
\end{theorem}
\begin{proof}
See Appendix \ref{app:proof_girsanov}.
\end{proof}

\section{Optimality Conditions in a Compatible P-Environment}
\label{sec:fixed_P_control}

In the usual fixed-flow formulation, the mean field term alone specifies the best-response problem.  Here it is random and unobserved; fixing it alone leaves its coupling with the initial state and driving noise undetermined.  We therefore fix a deterministic compatible joint law \(P\) of \((X,V,\mu)\), whose input marginal is
\begin{equation}
    \Gamma_P:=P\circ(X_0,V,\mu)^{-1}.
    \label{eq:fixed_P_input_law}
\end{equation}
The parameter \(P\) does not enlarge the controller's information: \(\mu\) remains unobserved and each admissible control has the WPR form \(u_t=\phi(t,y_t)\).  We derive the fixed-\(P\) response here and impose mean field consistency in Section~5.

\subsection{The Compatible Random Environment}

On a usual stochastic basis \((\Omega^P,\mathcal F^P,\mathbb F^P,\mathbb P_0^P)\), let \((X_0^P,V^P,\eta)\) have law \(\Gamma_P\), with \(V^P\) Brownian and \(\eta\) adapted in the compatible filtration.  Assumption (H1)(iii) is required only on the reference basis of Section~2, not on this enlarged basis. Let \(X^P=(x^P,y^P)\) be the pathwise unique solution of
\begin{equation}
    dX_t^P=\Sigma(t,X^P)dV_t^P,\qquad X_{t=0}^P=X_0^P .
    \label{eq:fixed_P_reference_state}
\end{equation}
For \(\phi\in\Phi_{ad}\), define
\begin{align}
    \theta_t^{P,\phi}
    &:=
    \Sigma^{-1}(t,X^P)
    F\bigl(t,X^P,\phi(t,y_t^P),\eta_t\bigr), \nonumber\\
    \Lambda_t^{P,\phi}
    &:=
    \mathcal E\!\left(
    \int_0^\cdot(\theta_s^{P,\phi})^\top dV_s^P
    \right)_t .
    \label{eq:fixed_P_likelihood}
\end{align}

The following condition specifies what it means to compare controls in the same prescribed random environment.

\textbf{Condition (EP) [Environment Preservation]:} For every \(\phi\in\Phi_{ad}\), \(\Lambda^{P,\phi}\) is a true martingale. Define the controlled probability measure and its Brownian motion by
\begin{equation}
    d\mathbb P^{P,\phi}
    =\Lambda_T^{P,\phi}d\mathbb P_0^P,\qquad
    V_t^{P,\phi}
    :=V_t^P-\int_0^t\theta_s^{P,\phi}ds,
    \label{eq:fixed_P_controlled_measure}
\end{equation}
Then \(V^{P,\phi}\) is Brownian under \(\mathbb P^{P,\phi}\), and
\begin{equation}
    \mathbb P^{P,\phi}\circ
    (X_0^P,V^{P,\phi},\eta)^{-1}
    =\Gamma_P .
    \label{eq:environment_preservation}
\end{equation}
Thus a deviation changes the controlled state law while the prescribed joint environment law \(\Gamma_P\) remains fixed.  Compatibility excludes future-noise leakage, whereas \emph{(EP)} supplies this invariance.  A compatible realization satisfying \emph{(EP)} is called an \emph{EP-compatible \(P\)-environment}; all results below are conditional on such a realization.

Fix \(p_0\ge4\), and write \(M_r(\nu):=\int_{\mathbb R^{d_x}}|z|^r\nu(dz)\).  We assume
\begin{equation}
    \int\left(
    |x_0|^{p_0}+\sup_{t\le T}M_{p_0}(\eta_t)
    \right)\Gamma_P(dx_0,dv,d\eta)<\infty .
    \label{eq:fixed_P_input_moment}
\end{equation}
The condition \(p_0\ge4\) ensures that quadratically growing cost terms are square-integrable; together with weak continuity of the input paths, the same bound yields their almost sure \(\mathcal W_2\)-continuity by uniform integrability of second moments. Under \(\mathbb P^{P,\phi}\), the reference coordinate solves
\begin{equation}
    dX_t^P
    =F\bigl(t,X^P,\phi(t,y_t^P),\eta_t\bigr)dt
    +\Sigma(t,X^P)dV_t^{P,\phi}.
    \label{eq:fixed_P_controlled_state}
\end{equation}
The corresponding objective is
\begin{equation}
    J_P(\phi):=
    \mathbb E^{\mathbb P^{P,\phi}}\!\left[
    \int_0^T l\bigl(t,X^P,\phi(t,y_t^P),\eta_t\bigr)dt
    +g(X_T^P)\right].
    \label{eq:fixed_P_cost}
\end{equation}

\begin{lemma}
\label{lem:fixed_P_controlled_moment}
Suppose Assumptions (H1)(i), (H2), and (H3) and Condition \emph{(EP)} hold, together with \eqref{eq:fixed_P_input_moment}.  Then a constant \(C_P<\infty\), independent of \(\phi\in\Phi_{ad}\), exists such that
\begin{equation}
    \sup_{\phi\in\Phi_{ad}}
    \mathbb E^{\mathbb P^{P,\phi}}\!\left[
    \|X^P\|_T^{p_0}+\sup_{t\le T}M_{p_0}(\eta_t)\right]
    \le C_P .
    \label{eq:fixed_P_controlled_moment}
\end{equation}
\end{lemma}
\begin{proof}
See Appendix \ref{app:proof_fixed_P_controlled_moment}.
\end{proof}

\subsection{The WPR Belief Measure and Variational System}

The conditional expectations in the maximum principle are taken with respect to the instantaneous \(\sigma\)-field \(\mathcal G_t^I=\sigma(y_t^P)\).  Since both the coefficients and the random input may depend on their histories, the required conditional object is a path law rather than a current state filter.

Let
\[
    \mathbf X_T:=C([0,T];\mathbb R^{d_x+d_y}),\qquad
    \mathbf M_T:=C([0,T];\mathcal P_2(\mathbb R^{d_x})).
\]
Equip them with the supremum norm and \(\sup_{r\le T}\mathcal W_2(\mu_r,\nu_r)\), respectively, and set \(\Xi:=\mathbf X_T\times\mathbf M_T\) and \(\widehat\Xi:=\mathbf X_T\times\mathbf M_T \times\mathbb R^{d_v}\), where \(d_v:=d_w+d_b=d_x+d_y\) by the square invertibility of \(\Sigma\) in Assumption (H2).  For either path space, define the nonanticipative truncation \(\mathsf S_tz\) by \((\mathsf S_tz)_r:=z_{r\wedge t}\), \(0\le r\le T\), and write \(\xi=(\xi^X,\xi^\mu,\xi^q)\in\widehat\Xi\) for the coordinates of the state path, mean field term, and adjoint integrand.

\begin{definition}[Fixed-\(P\) WPR belief measures]
\label{def:wpr_belief}
For \(\phi\in\Phi_{ad}\), let \(\pi_t^{P,\phi}\) be a regular conditional distribution of \((\mathsf S_tX^P,\mathsf S_t\eta)\) given \(\mathcal G_t^I\):
\begin{equation}
    \pi_t^{P,\phi}(A)
    :=
    \mathbb P^{P,\phi}\!\left(
    (\mathsf S_tX^P,\mathsf S_t\eta)\in A
    \mid\mathcal G_t^I\right),
    \qquad A\in\mathcal B(\Xi).
    \label{eq:fixed_P_wpr_belief}
\end{equation}
For every integrable Borel functional \(\Phi\) of \((\mathsf S_tX^P,\mathsf S_t\eta)\), the Kallianpur--Striebel formula reads
\begin{equation}
    \pi_t^{P,\phi}(\Phi)
    =
    \frac{\mathbb E^{\mathbb P_0^P}\!\left[
    \Lambda_t^{P,\phi}\Phi(\mathsf S_tX^P,\mathsf S_t\eta)
    \mid\mathcal G_t^I\right]}
    {\mathbb E^{\mathbb P_0^P}\!\left[
    \Lambda_t^{P,\phi}\mid\mathcal G_t^I\right]} .
    \label{eq:ks_eval}
\end{equation}

If \(\bar\phi^P\) is optimal and \(q^P\) is the Brownian integrand in \cref{thm:adjoint}, the extended WPR belief measure is
\begin{equation}
    \widehat\pi_t^{P,*}(B)
    :=
    \mathbb P^{P,\bar\phi^P}\!\left(
    (\mathsf S_tX^P,\mathsf S_t\eta,q_t^P)\in B
    \mid\mathcal G_t^I\right),
    \qquad B\in\mathcal B(\widehat\Xi).
    \label{eq:extended_wpr_belief}
\end{equation}
For \(dt\)-almost every \(t\) and almost every \(y\) under the observation marginal, this measure belongs to \(\mathcal P_2(\widehat\Xi)\) by \cref{lem:fixed_P_controlled_moment} and \(q^P\in L^2(dt\otimes d\mathbb P^{P,\bar\phi^P})\).
\end{definition}

The following theorem constructs the adjoint system using the Kunita--Watanabe decomposition.

\begin{theorem}
\label{thm:adjoint}
Suppose Assumptions (H0), (H1)(i), and (H2)--(H4), the strategy-class part of Assumption (H5), Condition \emph{(EP)}, and \eqref{eq:fixed_P_input_moment} hold.  Let \(\bar\phi^P\in\Phi_{ad}\) minimize \(J_P\). Then there are an adapted c\`adl\`ag process \(p^P\), a predictable \(\mathbb R^{d_v}\)-valued process \(q^P\), and a martingale \(N^P\) strongly orthogonal to every component of \(V^{P,\bar\phi^P}\), such that
\[
    \mathbb E^{\mathbb P^{P,\bar\phi^P}}\!\left[
    \sup_{t\le T}|p_t^P|^2+\int_0^T|q_t^P|^2dt+|N_T^P|^2
    \right]<\infty
\]
and
\begin{equation}
    dp_t^P
    =-l\bigl(t,X^P,\bar\phi^P(t,y_t^P),\eta_t\bigr)dt
    +(q_t^P)^\top dV_t^{P,\bar\phi^P}+dN_t^P,
    \qquad p_T^P=g(X_T^P).
    \label{eq:adjoint_bsde}
\end{equation}
Moreover, for every \(\psi\in\Phi_{ad}\),
\begin{equation}
\begin{split}
    0\le
    \mathbb E^{\mathbb P^{P,\bar\phi^P}}\!\int_0^T
    \Big\langle&
    F_1\bigl(t,X^P,\eta_t\bigr)^\top
    \Sigma^{-1}(t,X^P)^\top q_t^P\\
    &+\nabla_ul\bigl(t,X^P,\bar\phi^P(t,y_t^P),\eta_t\bigr),
    \psi(t,y_t^P)-\bar\phi^P(t,y_t^P)
    \Big\rangle dt .
\end{split}
    \label{eq:variational_ineq}
\end{equation}
\end{theorem}
\begin{proof}
See Appendix \ref{app:proof_adjoint}.
\end{proof}

Define the fixed-\(P\) Hamiltonian by
\begin{equation}
    \mathcal H(t,X,u,\eta,q)
    :=
    q^\top\Sigma^{-1}(t,X)F(t,X,u,\eta_t)
    +l(t,X,u,\eta_t).
    \label{eq:hamiltonian}
\end{equation}
For a selected optimizer \(\bar\phi^P\), set
\begin{equation}
    \widehat{\mathcal H}^{P}(t,y,u)
    :=
    \mathbb E^{\mathbb P^{P,\bar\phi^P}}\!\left[
    \mathcal H(t,X^P,u,\eta,q_t^P)\mid y_t^P=y\right].
    \label{eq:cond_hamiltonian}
\end{equation}
Assumptions (H0) and (H2)--(H4), together with \cref{lem:fixed_P_controlled_moment}, give
\[
    |\nabla_u\mathcal H(t,X^P,u,\eta,q_t^P)|
    \le C\left(
    1+|q_t^P|+\|X^P\|_t+
    \mathcal W_2(\eta_t,\delta_0)\right),
    \qquad u\in\mathbb U,
\]
whose right-hand side is integrable under \(dt\otimes d\mathbb P^{P,\bar\phi^P}\).  Conditional dominated convergence therefore yields
\[
    \nabla_u\widehat{\mathcal H}^{P}(t,y,u)
    =
    \mathbb E^{\mathbb P^{P,\bar\phi^P}}\!\left[
    \nabla_u\mathcal H(t,X^P,u,\eta,q_t^P)\mid y_t^P=y\right]
\]
for \(dt\)-almost every \(t\) and almost every \(y\) under the observation marginal.  Equivalently, \(\widehat{\mathcal H}^{P}(t,y,u)\) integrates \(\mathcal H\) against \(\widehat\pi_{t|y}^{P,*}\), averaging out \(\eta\). Thus the conditional Hamiltonian depends on the observation only through \(y\).

\begin{theorem}
\label{thm:wpr_smp}
Under the assumptions of \cref{thm:adjoint},
\begin{equation}
    \mathbb E^{\mathbb P^{P,\bar\phi^P}}\!\int_0^T
    \left\langle
    \nabla_u\widehat{\mathcal H}^{P}
    (t,y_t^P,\bar\phi^P(t,y_t^P)),
    \psi(t,y_t^P)-\bar\phi^P(t,y_t^P)
    \right\rangle dt\ge0
    \label{eq:smp_condition}
\end{equation}
for every \(\psi\in\Phi_{ad}\).
\end{theorem}
\begin{proof}
See Appendix \ref{app:proof_wpr_variational}.
\end{proof}

\subsection{Existence of the \texorpdfstring{Fixed-\(P\)}{Fixed-P} Best Response}

We now state the regularity needed to convert the integrated variational inequality into an admissible pointwise feedback.  Introduce
\begin{equation}
\begin{split}
    c_{\mathcal H,t}(\xi,\zeta)
    &:=(1+|\xi^q|+|\zeta^q|)
    \left(\|\xi^X-\zeta^X\|_t
    +\sup_{0\le r\le t}
    \mathcal W_2(\xi_r^\mu,\zeta_r^\mu)\right)
    +|\xi^q-\zeta^q|,\\
    \mathcal C_{\mathcal H}^t(\rho,\rho')
    &:=\inf_{\Gamma\in\Pi(\rho,\rho')}
    \int_{\widehat\Xi\times\widehat\Xi}
    c_{\mathcal H,t}(\xi,\zeta)\,\Gamma(d\xi,d\zeta).
\end{split}
    \label{eq:hamiltonian_gradient_cost}
\end{equation}

\begin{lemma}
\label{lem:pointwise_hamiltonian_gradient}
Suppose Assumptions (H0) and (H2)--(H4) hold.  There is \(L_{\mathcal H}>0\) such that, for all \(t\in[0,T]\), \(u\in\mathbb U\), and \(\xi,\zeta\in\widehat\Xi\),
\begin{equation}
    \left|
    \nabla_u\mathcal H(t,\xi^X,u,\xi^\mu,\xi^q)
    -\nabla_u\mathcal H(t,\zeta^X,u,\zeta^\mu,\zeta^q)
    \right|
    \le L_{\mathcal H}c_{\mathcal H,t}(\xi,\zeta).
    \label{eq:pointwise_hamiltonian_gradient}
\end{equation}
\end{lemma}
\begin{proof}
See Appendix \ref{app:proof_pointwise_hamiltonian_gradient}.
\end{proof}

Choose a version of the extended WPR belief measure in \eqref{eq:extended_wpr_belief} that is jointly Borel in \((t,y)\), using disintegration on \(dt\otimes d\mathbb P^{P,\bar\phi^P}\).

\textbf{Assumption (H6) [Regularity of the Fixed-\(P\) Extended WPR Belief Measure]:}

\((\mathrm{i})\) There is \(C_\pi>0\) such that
\begin{equation}
    \mathcal C_{\mathcal H}^t
    (\widehat\pi_{t|y}^{P,*},\widehat\pi_{t|y'}^{P,*})
    \le C_\pi|y-y'|
    \label{eq:H6_belief_observation_stability}
\end{equation}
for all \(t\in[0,T]\) and \(y,y'\in\mathbb R^{d_y}\).

\((\mathrm{ii})\) There is \(C_m>0\) such that
\begin{equation}
    \int_{\widehat\Xi}
    \left(1+\|\xi^X\|_t^2+|\xi^q|^2
    +\sup_{0\le r\le t}
    \mathcal W_2^2(\xi_r^\mu,\delta_0)\right)
    \widehat\pi_{t|y}^{P,*}(d\xi)
    \le C_m(1+|y|^2).
    \label{eq:H6_conditional_moment}
\end{equation}

\((\mathrm{iii})\) Equip \(\Phi_{ad}\) with its compact-open Borel \(\sigma\)-field.  A jointly Borel version
\[
    D^P(t,y,\phi)
    :=\mathbb E^{\mathbb P_0^P}\!\left[
      \Lambda_t^{P,\phi}\mid y_t^P=y\right]
\]
can be selected on \([0,T]\times\mathbb R^{d_y}\times\Phi_{ad}\), and there is \(\underline d_\Lambda>0\) such that
\begin{equation}
    D^P(t,y,\phi)
    \ge\underline d_\Lambda
    \label{eq:H6_bayes_denominator}
\end{equation}
for all \(t,y\) and \(\phi\in\Phi_{ad}\).  This lower bound fixes a pointwise Bayes version in \eqref{eq:ks_eval}; the almost-everywhere identities require only positivity.

\((\mathrm{iv})\) With
\[
    \Psi^P(t,y,u):=
    \int_{\widehat\Xi}
    \nabla_u\mathcal H(t,\xi^X,u,\xi^\mu,\xi^q)
    \widehat\pi_{t|y}^{P,*}(d\xi),
\]
there is \(C_{\rm time}>0\) such that
\begin{equation}
    |\Psi^P(t,y,u)-\Psi^P(s,y,u)|
    \le C_{\rm time}|t-s|
    \label{eq:H6_time_gradient_stability}
\end{equation}
for all \(s,t,y,u\).  The selected versions satisfy \((\mathrm{i})\)--\((\mathrm{iv})\) also outside the observation support.

\begin{lemma}
\label{lem:conditional_gradient_stability}
Under Assumptions (H0), (H2)--(H4), and parts (i), (ii), and (iv) of Assumption (H6),
\begin{equation}
    |\Psi^P(t,y,u)-\Psi^P(s,y',u)|
    \le L_\Psi(|t-s|+|y-y'|),
    \qquad
    L_\Psi:=\max\{C_{\rm time},L_{\mathcal H}C_\pi\}.
    \label{eq:H6_observation_gradient_stability}
\end{equation}
\end{lemma}
\begin{proof}
See Appendix \ref{app:proof_conditional_gradient_stability}.
\end{proof}

\begin{theorem}
\label{thm:fixed_P_existence}
Suppose Assumptions (H0), (H1)(i), and (H2)--(H4), the strategy-class part of Assumption (H5), Condition \emph{(EP)}, and \eqref{eq:fixed_P_input_moment} hold.  Then there exists \(\bar\phi^P\in\Phi_{ad}\) such that
\begin{equation}
    J_P(\bar\phi^P)
    =\inf_{\phi\in\Phi_{ad}}J_P(\phi).
    \label{eq:fixed_P_optimizer}
\end{equation}
Every optimizer satisfies \eqref{eq:smp_condition}.  Fix one optimizer \(\bar\phi^P\), its adjoint integrand \(q^P\), and a jointly Borel version of the associated extended WPR belief measure.  If this selected version satisfies parts (i), (ii), and (iv) of Assumption (H6) and \(L_\Phi\ge L_\Psi/\lambda\), the corresponding conditional Hamiltonian has the unique pointwise minimizer
\begin{equation}
    \phi^P(t,y)
    :=\arg\min_{u\in\mathbb U}
    \widehat{\mathcal H}^{P}(t,y,u).
    \label{eq:pointwise_feedback}
\end{equation}
It belongs to \(\Phi_{ad}\), and its induced control agrees with the selected optimizer \(\bar\phi^P(t,y_t^P)\) \(dt\otimes d\mathbb P^{P,\bar\phi^P}\)-almost everywhere.  Thus the pointwise selector is unique for the selected version of the extended WPR belief measure and induces the same optimal controlled probability measure and cost as \(\bar\phi^P\).
\end{theorem}
\begin{proof}
See Appendix \ref{app:proof_fixed_P_existence}.
\end{proof}

\begin{theorem}
\label{thm:lipschitz}
Under the assumptions of \cref{thm:fixed_P_existence}, including parts (i), (ii), and (iv) of Assumption (H6) and \(L_\Phi\ge L_\Psi/\lambda\), the resulting strategy satisfies
\begin{equation}
    |\phi^P(t,y)-\phi^P(s,y')|
    \le\frac{L_\Psi}{\lambda}
    (|t-s|+|y-y'|)
    \label{eq:fixed_P_selector_lipschitz}
\end{equation}
for all \(s,t\in[0,T]\) and \(y,y'\in\mathbb R^{d_y}\).
\end{theorem}
\begin{proof}
This is \eqref{eq:fixed_P_selector_estimate}, established in Step~2 of Appendix \ref{app:proof_fixed_P_existence}.
\end{proof}

\section{Conditional Mean Field Law and Its Path-Posterior KS Representation}
\label{sec:conditional_mean_field}

Fix an EP-compatible \(P\)-environment and the selected optimal strategy \(u_t^*=\phi^P(t,y_t)\) from \cref{thm:fixed_P_existence}.  Write \(\mathbb P^*:=\mathbb P^{P,\phi^P}\), suppress the superscript \(P\), and retain \(\eta\) for the random input.  Let \(\mathbb F^Y\) be the usual augmentation of \(\sigma(y_s:0\le s\le t)\).  Equivalence of \(\mathbb P^*\) and \(\mathbb P_0^P\) on \(\mathcal F_T^P\) gives the same completed observation filtration.

Unlike the WPR belief measure of Section 3, the conditional population law uses \(\mathcal F_t^Y\), whereas the control reads only \(\mathcal G_t^I=\sigma(y_t)\).  Since the coefficients depend on both hidden state history and \(\eta\), we first construct their joint path posterior and then take its current state marginal.

\subsection{Existence of the Conditional Mean Field Law}

For \(t\in[0,T]\), \(A\in\mathcal B(\mathbb R^{d_x})\), and bounded Borel \(\varphi:\mathbb R^{d_x}\to\mathbb R\), set {\small
\begin{equation}
    m_t^Y(A):=\mathbb P^*(x_t^*\in A\mid\mathcal F_t^Y),
    \qquad
    m_t^Y(\varphi):=\mathbb E^{\mathbb P^*}
    [\varphi(x_t^*)\mid\mathcal F_t^Y].
    \label{eq:conditional_marginal_law}
\end{equation}
}
\noindent On the Polish space \(\mathsf E:=C([0,T];\mathbb R^{d_x})\times\mathbf M_T\), set \(Z:=(x^*,\eta)\).  For each fixed \(t\), Ji\v{r}ina's theorem \cite{jirina1959regular,kallenberg2002foundations} gives the regular conditional law
\[
    \Pi_t(\Phi)
    :=\mathbb E^{\mathbb P^*}[\Phi(Z)\mid\mathcal F_t^Y],
    \qquad \Phi\in\mathcal B_b(\mathsf E).
\]
Its restriction to the coordinates up to \(t\) is the posterior of the hidden histories.  With \(e_t(\xi,\zeta):=\xi(t)\), {\small
\begin{equation}
    m_t^Y=\Pi_t\circ e_t^{-1},\qquad
    m_t^Y(\varphi)=\Pi_t(\varphi\circ e_t).
    \label{eq:marginal_of_path_filter}
\end{equation}
}

As shown in Appendix \ref{app:proof_joint_posterior_innovation}, an integrable inf-compact function and a countable convergence-determining class yield unique weak right and left limits from a dense set of times.  Thus \(\{\Pi_t\}_{t\le T}\) has an adapted weakly c\`adl\`ag, hence optional, version, and outside one common evanescent set,
\[
    \Pi_t\!\left(
    \|\xi\|_T^2+\sup_{s\le T}M_2(\zeta_s)\right)<\infty
\]
The Portmanteau theorem applied to the conditioned inf-compact functional gives this bound.  Hence \(m^Y\) is \(\mathcal P_2(\mathbb R^{d_x})\)-valued, and at each fixed time it agrees almost surely with every regular conditional version on bounded continuous tests.

\subsection{Weak KS Representation of the Conditional Mean Field Law}

For \((\xi,\zeta)\in\mathsf E\), abbreviate {\small
\[
\begin{aligned}
    b_t^*(\xi,\zeta)
    &:=b(t,\xi,y_{\cdot\wedge t},u_t^*,\zeta_t),
    &h_t^*(\xi,\zeta)
    &:=h(t,\xi,y_{\cdot\wedge t},u_t^*,\zeta_t),\\
    \Gamma_t(\xi)
    &:=\sigma_1(t,\xi,y_{\cdot\wedge t})
      \sigma_1(t,\xi,y_{\cdot\wedge t})^\top,
    &a_t
    &:=\sigma_2(t,y_{\cdot\wedge t})
      \sigma_2(t,y_{\cdot\wedge t})^\top .
\end{aligned}
\]
}
By Assumptions (H0)--(H2), \(a\) is predictable and uniformly positive definite; \(a_t^{-1/2}\) denotes its symmetric positive definite inverse square root.  For \(\varphi\in C_b^2(\mathbb R^{d_x})\), define
\begin{equation}
    [\mathcal A_t^{u^*}\varphi](\xi,\zeta)
    :=b_t^*(\xi,\zeta)\cdot\nabla\varphi(\xi(t))
      +\frac12\operatorname{Tr}
       [\Gamma_t(\xi)\nabla^2\varphi(\xi(t))].
    \label{eq:path_dependent_generator}
\end{equation}

\begin{lemma}
\label{lem:joint_posterior_innovation}
Under the assumptions of \cref{thm:fixed_P_existence}, \(\Pi\) admits an optional version.  Let
\[
 f:[0,T]\times C([0,T];\mathbb R^{d_y})\times\mathsf E\to\mathbb R^k
\]
be jointly Borel and nonanticipative, and put \(f_t(z):=f(t,y_{\cdot\wedge t},z)\).  If
\[
 \mathbb E^{\mathbb P^*}\int_0^T|f_t(Z)|dt<\infty,
\]
then \(\Pi_t(f_t)\) has a predictable representative, defined \(dt\otimes d\mathbb P^*\)-almost everywhere.  With \(\widehat h_t:=\Pi_t(h_t^*)\),
\begin{equation}
    \nu_t:=\int_0^t a_s^{-1/2}(dy_s-\widehat h_sds)
    \label{eq:innovation_process}
\end{equation}
is an \(\mathbb F^Y\)-Brownian motion under \(\mathbb P^*\), and \(\mathbb F^Y\) has the predictable representation property with respect to \(\nu\).
\end{lemma}
\begin{proof}
See Appendix \ref{app:proof_joint_posterior_innovation}.
\end{proof}

For \(\varphi\in C_b^2(\mathbb R^{d_x})\), It\^o's formula, conditional Fubini, and optional projection make
\[
    M_t^\varphi:=m_t^Y(\varphi)-m_0^Y(\varphi)
    -\int_0^t\Pi_s(\mathcal A_s^{u^*}\varphi)ds
\]
square-integrable and \(\mathbb F^Y\)-martingale.  By \cref{lem:joint_posterior_innovation}, \(dM_t^\varphi=K_t^\varphi d\nu_t\) for a predictable row-vector process \(K^\varphi\).  Set
\begin{equation*}
    \mathcal C_t^\varphi
    :=\Pi_t\!\left((\varphi\circ e_t)(h_t^*)^\top\right)
      -m_t^Y(\varphi)\Pi_t((h_t^*)^\top).
\end{equation*}
Projecting It\^o's formula for \(\varphi(x_t^*)y_t^\top\) and comparing it with the product formula for \(m_t^Y(\varphi)y_t^\top\) gives \(K_t^\varphi a_t^{1/2}=\mathcal C_t^\varphi\) \(dt\otimes d\mathbb P^*\)-almost everywhere.  Independence of \(W^{u^*}\) and \(B^{u^*}\) removes the signal--observation covariation, and \(p_0\ge4\) justifies the projections.  Therefore
\begin{equation}
\begin{aligned}
    dm_t^Y(\varphi)
    &=\Pi_t(\mathcal A_t^{u^*}\varphi)dt
      +\mathcal C_t^\varphi a_t^{-1/2}d\nu_t\\
    &=\Pi_t(\mathcal A_t^{u^*}\varphi)dt
      +\mathcal C_t^\varphi a_t^{-1}
      \big(dy_t-\Pi_t(h_t^*)dt\big).
\end{aligned}
    \label{eq:path_dependent_KS}
\end{equation}
For each fixed \(\varphi\), the integral form of \eqref{eq:path_dependent_KS} holds up to indistinguishability.  It need not close in \(m_t^Y\) alone, since \(\Pi_t\) retains the hidden history and the random mean field term; no \(\mathcal W_2\)-path continuity is asserted.

\section{Existence of WPR-MFG Equilibria}
\label{sec:equilibrium}

This section links the fixed-\(P\) best response of Section 3 to the conditional mean field law of Section 4.  In common noise MFGs, a random mean field term adapted to the increasing common-information filtration may enter the control; see \cite{carmona2016common}.  Under WPR, sampling \(\mu\) and then applying \(u_t^{*,\mu}=\phi^\mu(t,y_t)\) generally produces a \(\sigma(\mu,y_t)\)-measurable control.  It is not, in general, \(\sigma(y_t)\)-measurable.

To preserve WPR, we parameterize the feedback by a deterministic joint law \(P\) and seek a compatible conditional McKean--Vlasov response whose mean field coordinate both enters the state equation and equals the conditional state law.  Sections 3--4 supply the selected feedback and conditional-law update; their recursive closure defines the law-level response used below.

For \(r\ge1\), set \(M_r(\eta):=\int_{\mathbb R^{d_x}}|x|^r\eta(dx)\). Let \(d_{\rm w}\) be a bounded complete metric for weak convergence on \(\mathcal P(\mathbb R^{d_x})\).  Equip \(\overline{\mathcal Z}:=C([0,T];\mathcal P(\mathbb R^{d_x}))\) with \(\rho_{\mathcal Z}(\mu,\nu):=\sup_{t\le T}d_{\rm w}(\mu_t,\nu_t)\), and set
\[
 \mathcal Z:=\big\{\mu\in\overline{\mathcal Z}:
 \mu_t\in\mathcal P_2(\mathbb R^{d_x})\text{ for every }t\big\}.
\]
On sets with a uniform \(p_0\)-moment bound for some \(p_0>2\), \(\rho_{\mathcal Z}\)-convergence implies uniform-in-time \(\mathcal W_2\)-convergence.

Set {\small
\begin{equation}
 \mathcal X:=C([0,T];\mathbb R^{d_x+d_y}),\quad
 \mathcal V:=C([0,T];\mathbb R^{d_w+d_b}),\quad
 \Omega^{\mathrm c}:=\mathcal X\times\mathcal V\times\overline{\mathcal Z}.
 \label{eq:canonical_joint_space}
\end{equation}
}
The canonical coordinates are denoted by \((X,V,\mu)\), where \(X=(x,y)\) and \(V=(W,B)\).  For joint laws with finite second moments, let {\small
\begin{equation}
 \mathbf d_t^2(P,P'):=\inf_{\Gamma\in\Pi(P,P')}\int
 \left(\|X-X'\|_t^2+\|V-V'\|_t^2+
 \sup_{r\le t}\mathcal W_2^2(\mu_r,\mu_r')\right)d\Gamma .
 \label{eq:joint_law_distance}
\end{equation}
}

\begin{definition}[Compatible joint laws and compatible inputs]
\label{def:compatible_joint_laws}
Let \(\lambda_0\) be the law of \(X_0\), and let \(\mathbb W_V\) be Wiener measure on \(\mathcal V\).  For \(t\in[0,T]\), set
\[
 \mathcal H_t:=\sigma(X_s,V_s,\mu_s:0\le s\le t),\qquad
 \mathcal F_t^{X_0,V}:=\sigma(X_0,V_s:0\le s\le t).
\]
A joint law \(P\in\mathcal P(\Omega^{\mathrm c})\) of \((X,V,\mu)\) is called compatible if \(P\circ(X_0,V)^{-1}=\lambda_0\otimes\mathbb W_V\) and, for every bounded Borel function \(H\) of \((X_0,V)\) and every \(t\in[0,T]\),
\[
 \mathbb E^P[H(X_0,V)\mid\mathcal H_t]
 =\mathbb E^P[H(X_0,V)\mid\mathcal F_t^{X_0,V}].
\]
The collection of such laws is denoted by \(\mathfrak K\).  Given an input flow \(\eta\), a stochastic basis carrying \((X_0,V,\eta)\) is called compatible with \(\eta\) if the same identity holds with \(\sigma(X_0,V_s,\eta_s:0\le s\le t)\) in place of \(\mathcal H_t\).
\end{definition}

Fix
\[
 p_0\ge4,\qquad 2<p_{\rm tm}\le p_0,\qquad
 \alpha_{\rm tm}:=\frac{p_{\rm tm}}2-1>0,
\]
and write
\begin{equation}
 \mathcal M_{p_0}(P):=\mathbb E^P\!\left[
 \|X\|_T^{p_0}+\sup_{t\le T}M_{p_0}(\mu_t)\right].
 \label{eq:joint_moment_radius}
\end{equation}
The distance \(\mathbf d_T\) remains quadratic.  The exponent \(p_0\) supplies uniform integrability, while \(p_{\rm tm}\) is used only for path tightness. For \(P\in\mathfrak K\), the \(P\)-environment control problem is the problem studied in Section 3.  On a compatible weak basis, the joint law of \((X_0,V,\eta)\) is \(P\circ(X_0,V,\mu)^{-1}\), while the strategy retains the WPR form \(u_t=\phi(t,y_t)\).  Theorem \ref{thm:fixed_P_existence} supplies a selected optimal controlled probability measure, denoted by \(\mathbb P^{P,*}:=\mathbb P^{P,\phi^P}\), and a deterministic feedback \(\phi^P\in\Phi_{ad}\).  Let \(q^P\) be the \(V\)-integrand in the associated adjoint equation.  Set
\begin{equation}
\begin{split}
 \widehat{\mathcal H}^{P}(t,y,u)
 &:=\mathbb E^{\mathbb P^{P,*}}\!\left[
 \mathcal H(t,X^P,u,\eta,q_t^P)\mid y_t^P=y\right],\\
 \phi^P(t,y)&:=\arg\min_{u\in\mathbb U}
 \widehat{\mathcal H}^{P}(t,y,u).
\end{split}
 \label{eq:law_level_selector}
\end{equation}

A compatible \(P\)-response is a law \(Q\in\mathfrak K\) with \(\mathcal M_{p_0}(Q)<\infty\) such that, under \(Q\),
\begin{equation}
 dX_t=F\big(t,X,\phi^P(t,y_t),\mu_t\big)dt+\Sigma(t,X)dV_t
 \label{eq:joint_law_response_system}
\end{equation}
and
\begin{equation}
 \mu_t=Q(x_t\in\cdot\mid\mathcal F_t^Y),
 \qquad 0\le t\le T,\quad Q\text{-a.s.}
 \label{eq:response_conditional_consistency}
\end{equation}
The latter identity is understood through bounded continuous state tests and bounded \(\mathcal F_t^Y\)-measurable tests.  The construction below gives a unique law in the compatible class covered by Assumption (H7)(ii).  Denote it by \(Q^P\) and define
\begin{equation}
 \mathcal R(P):=Q^P.
 \label{eq:joint_response_map}
\end{equation}

\begin{definition}[WPR-MFG equilibria]
A weak WPR-MFG equilibrium is a fixed point \(P^*\in\mathfrak K\) of the joint law response for which \(\mathcal R(P^*)\) is defined and
\begin{equation}
 \mathcal R(P^*)=P^*.
 \label{eq:joint_law_fixed_point}
\end{equation}
Under \(P^*\), the canonical coordinates satisfy
\begin{equation}
 dX_t=F\big(t,X,\phi^{P^*}(t,y_t),\mu_t\big)dt+\Sigma(t,X)dV_t,
 \qquad
 \mu_t=P^*(x_t\in\cdot\mid\mathcal F_t^Y).
 \label{eq:weak_equilibrium_consistency}
\end{equation}
The equilibrium control is \(u_t^*=\phi^{P^*}(t,y_t)\).

A strong realization of a WPR-MFG equilibrium on a compatible stochastic basis with controlled probability measure \(\mathbb P^*\) is a pair \((X^*,\mu^*)\) such that, with \(P^*:=\mathbb P^*\circ(X^*,V,\mu^*)^{-1}\), the control \(u_t^*=\phi^{P^*}(t,y_t^*)\) is optimal and \(\mathcal G_t^I\)-measurable, \(X^*\) is a strong solution of the closed-loop state equation, and
\begin{equation}
 \mu_t^*=\mathbb P^*(x_t^*\in\cdot\mid\mathcal F_t^Y),
 \qquad 0\le t\le T,\quad\mathbb P^*\text{-a.s.}
 \label{eq:strong_equilibrium_consistency}
\end{equation}
\end{definition}

To obtain a continuous law-level response from the fixed-\(P\) optimizer, we need estimates uniform in both the law parameter and the conditional mean field input.

\textbf{Assumption (H7) [Law-level stability and regularity]:}

Throughout, \(P\in\mathfrak K\), \(\mathcal M_{p_0}(P)<\infty\), and the input flow \(\eta\) has continuous \(\mathcal P_2(\mathbb R^{d_x})\)-valued paths, is compatible with \((X_0,V)\), and satisfies \(\mathbb E[\sup_{t\le T}M_{p_0}(\eta_t)]<\infty\). The assumptions of Section 3 hold uniformly over this class of compatible \(P\)-environments.  Hence \cref{thm:fixed_P_existence,thm:lipschitz} apply with constants independent of \(P\) and \(\eta\).  This uniformity includes the likelihood estimate used in Assumption (H5): for every compatible reference realization associated with \((P,\eta,\phi)\), write \(\mathbb P_0^{P,\eta}\) and \(\Lambda^{P,\eta,\phi}\) for its reference law and likelihood process.  There exist \(p_\Lambda>1\) and \(C_\Lambda<\infty\), independent of \((P,\eta,\phi)\), such that
\begin{equation}
 \mathbb E^{\mathbb P_0^{P,\eta}}\!\left[
 (\Lambda_T^{P,\eta,\phi})^{p_\Lambda}\right]\le C_\Lambda.
 \label{eq:H7_uniform_likelihood}
\end{equation}

\((\mathrm{i})\) The conditional Hamiltonian gradients associated with the selected optimal responses in \cref{thm:fixed_P_existence} admit versions, denoted by
\[
 \Psi^P(t,y,u):=\nabla_u\widehat{\mathcal H}^{P}(t,y,u)
\]
such that \((P,t,y,u)\mapsto\Psi^P(t,y,u)\) is jointly Borel.  There are \(\beta\in(0,1]\) and \(L_H>0\) such that
\begin{equation}
\begin{split}
 |\Psi^P(t,y,u)-\Psi^{P'}(s,y',u)|
 \le L_H\!\left(
 |t-s|^\beta+|y-y'|
 +(1+|y|+|y'|)\mathbf d_T(P,P')\right).
\end{split}
 \label{eq:H7_hamiltonian_stability}
\end{equation}

\((\mathrm{ii})\) For fixed \(P\) and \(\eta\), denote the controlled state law by \(\mathbb Q^{P,\eta}\), and set \(m_t^{P,\eta}:=\mathbb Q^{P,\eta} (x_t^{P,\eta}\in\cdot\mid\mathcal F_t^{y^{P,\eta}})\).  We choose a continuous version of \(m^{P,\eta}\) that is adapted to \(\{\mathcal F_t^{y^{P,\eta}}\}_{t\le T}\) and representable as a nonanticipative Borel functional of \(y_{\cdot\wedge t}^{P,\eta}\).  There is \(L_{\rm fil}>0\), independent of the law parameters and of the step in the recursive construction, such that every pair of these systems constructed on the same compatible basis with the same \((X_0,V)\) satisfies
\begin{equation}
 \mathbb E\!\left[\sup_{s\le t}
 \mathcal W_2^2(m_s^1,m_s^2)\right]
 \le L_{\rm fil}\mathbb E\!\left[
 \sup_{s\le t}|y_s^1-y_s^2|^2\right],
 \qquad t\in[0,T].
 \label{eq:H7_input_filter_stability}
\end{equation}
The admissible comparisons include adjacent approximations, an approximation and its limit, finite-step systems with different \(P\), and any jointly compatible pair of self-consistent responses used in the uniqueness argument. The same bound is required after the compatible couplings invoked below.

\((\mathrm{iii})\) There is \(C_{\rm tm}>0\) such that every conditional-law output used in the recursive construction, including its limit, satisfies
\begin{equation}
 \mathbb E\!\left[
 \mathcal W_2^{p_{\rm tm}}(m_t^{P,\eta},m_s^{P,\eta})\right]
 \le C_{\rm tm}|t-s|^{1+\alpha_{\rm tm}},
 \qquad s,t\in[0,T].
 \label{eq:H7_filter_time_modulus}
\end{equation}

With \(q_\Lambda=p_\Lambda/(p_\Lambda-1)\), conditional Jensen, Doob's inequality, and \eqref{eq:H7_uniform_likelihood} yield the uniform bound
\begin{equation}
 \sup_{P,\eta}\mathbb E^{\mathbb Q^{P,\eta}}\!\left[
 \sup_{t\le T}M_{p_0}(m_t^{P,\eta})\right]
 \le C_{\rm cond}<\infty.
 \label{eq:derived_conditional_moment_bound}
\end{equation}
The calculation is included in Appendix \ref{app:proof_joint_response_wellposed}.

Strong convexity and (H7)(i) give the selector estimate
\begin{equation}
\begin{split}
 |\phi^P(t,y)-\phi^{P'}(s,y')|
 \le \frac{L_H}{\lambda}\!\left(
 |t-s|^\beta+|y-y'|
 +(1+|y|+|y'|)\mathbf d_T(P,P')\right).
\end{split}
 \label{eq:H7_selector_stability}
\end{equation}
This is the only place where the selected optimal response enters the fixed-point argument.

For the construction below, let \(K_1,K_2<\infty\) initially be arbitrary radii satisfying
\begin{equation}
 K_1\ge\mathbb E|X_0|^{p_0},\qquad K_2>0,
 \label{eq:candidate_radii}
\end{equation}
and define
\begin{equation}
\label{eq:W_definition}
 \mathfrak W:=
 \left\{
 \begin{array}{@{}l@{}}
 P\in\mathfrak K:\ \mathcal M_{p_0}(P)\le K_1,\\[0.5em]
 \displaystyle
 \mathbb E^P\!\left[
 |X_t-X_s|^{p_{\rm tm}}
 +\mathcal W_2^{p_{\rm tm}}(\mu_t,\mu_s)\right]
 \le K_2|t-s|^{1+\alpha_{\rm tm}},\quad s,t\in[0,T]
 \end{array}
 \right\}.
\end{equation}

\begin{lemma}
\label{lem:closed_loop_input_stability}
Under (H0)--(H3) and the strategy-class part of (H5), fix \(\phi\in\Phi_{ad}\).  Let \(\eta\) be a compatible random input with values in \(C([0,T];\mathcal P_2(\mathbb R^{d_x}))\), and suppose that \(\mathbb E[\sup_{t\le T}M_2(\eta_t)]<\infty\).  Then the closed-loop state equation with feedback \(\phi(t,y_t)\) has a pathwise unique strong solution \(X^{\phi,\eta}\).  Write
\[
 \mathcal E_{\rm in}:=\mathbb R^{d_x+d_y}\times\mathcal V\times
 C([0,T];\mathcal P_2(\mathbb R^{d_x})),
 \qquad \gamma:=\mathcal L(X_0,V,\eta).
\]
There is a Borel map \(S_{\phi,\gamma}:\mathcal E_{\rm in}\to\mathcal X\), nonanticipative on a Borel set of full \(\gamma\)-measure, such that every compatible realization with environment law \(\gamma\) satisfies
\[
 X^{\phi,\eta}=S_{\phi,\gamma}(X_0,V,\eta)\qquad\text{a.s.}
\]
Consequently, the law \(\mathcal L(X^{\phi,\eta},X_0,V,\eta)\) is unique for fixed \((\phi,\gamma)\).  If \(\eta,\eta'\) are placed on a jointly compatible basis with the same \((X_0,V)\), meaning that the compatibility identity remains valid after both input histories are adjoined, then
\begin{equation}
 \mathbb E\!\left[
 \|X^{\phi,\eta}-X^{\phi,\eta'}\|_t^2\right]
 \le C_{\rm st}\int_0^t
 \mathbb E\!\left[\sup_{r\le s}
 \mathcal W_2^2(\eta_r,\eta_r')\right]ds,
 \label{eq:closed_loop_input_stability}
\end{equation}
where \(C_{\rm st}\) is independent of \(\phi\in\Phi_{ad}\).  For the selected feedback, we write \(X^{P,\eta}:=X^{\phi^P,\eta}\); the estimate is therefore uniform over \(P\in\mathfrak W\).
\end{lemma}
\begin{proof}
See Appendix \ref{app:proof_closed_loop_input_stability}.
\end{proof}

\begin{lemma}
\label{lem:compatible_law_closed}
The class \(\mathfrak K\) is convex and closed under weak convergence.
\end{lemma}
\begin{proof}
See Appendix \ref{app:proof_compatible_law_closed}.
\end{proof}

\begin{lemma}
\label{lem:joint_response_wellposed}
Under Assumptions (H0)--(H7), every \(P\in\mathfrak W\) has a compatible \(P\)-response.  Its law is unique in the class specified in Assumption (H7)(ii), and it has a strong realization on the common compatible stochastic basis used in the construction.  Thus \(\mathcal R(P)\) is well defined, and \(u_t^P=\phi^P(t,y_t)\) is an admissible WPR strategy.
\end{lemma}
\begin{proof}
See Appendix \ref{app:proof_joint_response_wellposed}.
\end{proof}

The response law is now well defined.  We next construct a compact domain that it preserves.

\begin{lemma}
\label{lem:W_compact}
The set \(\mathfrak W\) is nonempty, convex, and compact in \(\mathcal P(\Omega^{\mathrm c})\) endowed with the weak topology.
\end{lemma}
\begin{proof}
See Appendix \ref{app:proof_W_compact}.
\end{proof}

To prove invariance, we control the moment and time-modulus terms in \eqref{eq:W_definition}.

\begin{lemma}
\label{lem:response_estimates}
Under Assumptions (H0)--(H7), there are finite constants \(A_{p_0}(T)\), \(B_X(T),B_\mu(T)\), independent of \(P\), such that
\begin{align}
 \mathcal M_{p_0}(\mathcal R(P))&\le A_{p_0}(T),
 \label{eq:response_moment_bound}\\
 \mathbb E^{\mathcal R(P)}|X_t-X_s|^{p_{\rm tm}}
 &\le B_X(T)|t-s|^{1+\alpha_{\rm tm}},
 \label{eq:response_state_modulus}\\
 \mathbb E^{\mathcal R(P)}
 \mathcal W_2^{p_{\rm tm}}(\mu_t,\mu_s)
 &\le B_\mu(T)|t-s|^{1+\alpha_{\rm tm}}.
 \label{eq:response_modulus_bound}
\end{align}
\end{lemma}
\begin{proof}
See Appendix \ref{app:proof_response_estimates}.
\end{proof}

The constants above are independent of the finite radii in \eqref{eq:candidate_radii}; choose
\begin{equation}
 K_1\ge\max\{A_{p_0}(T),\mathbb E|X_0|^{p_0}\},\qquad
 K_2\ge B_X(T)+B_\mu(T).
 \label{eq:fixed_point_radii}
\end{equation}
Then \cref{lem:joint_response_wellposed} and \eqref{eq:response_moment_bound}--\eqref{eq:response_modulus_bound} give \(\mathcal R(\mathfrak W)\subseteq\mathfrak W\).  Having obtained a compact invariant domain, it remains to prove continuity of the response on it.

\begin{lemma}
\label{lem:T_continuity}
Under Assumptions (H0)--(H7), the map \(\mathcal R:\mathfrak W\to\mathfrak W\) is continuous for the weak topology.
\end{lemma}
\begin{proof}
See Appendix \ref{app:proof_T_continuity}.
\end{proof}

Before applying the fixed-point theorem, we record how a fixed response law is identified with the optimizer selected in Section 3.

\begin{lemma}
\label{lem:fixed_point_optimal_response_identification}
Suppose \(P\in\mathfrak W\) and \(\mathcal R(P)=P\).  Then
\[
 \mathbb P^{P,*}\circ
 \bigl(X^P,V^{P,\phi^P},\eta\bigr)^{-1}=P .
\]
Consequently, \(\phi^P(t,y_t)\) is an optimal WPR response in the fixed-point realization and in every strong compatible realization having joint law \(P\).
\end{lemma}
\begin{proof}
See Appendix \ref{app:proof_fixed_point_optimal_response_identification}.
\end{proof}

We now have a continuous self-map of a nonempty compact convex set, so the Schauder--Tychonoff theorem applies.

\begin{theorem}
\label{thm:weak_equilibrium}
Suppose Assumptions (H0)--(H7) hold.  Then there exists a weak WPR-MFG equilibrium.
\end{theorem}
\begin{proof}
See Appendix \ref{app:proof_weak_equilibrium}.
\end{proof}

The compatible response construction upgrades the fixed point to a strong realization.

\begin{theorem}
\label{thm:strong_equilibrium}
Suppose Assumptions (H0)--(H7) hold.  Every weak equilibrium obtained in \cref{thm:weak_equilibrium} admits a strong compatible realization.  For a fixed equilibrium law \(P^*\) and the feedback \(\phi^{P^*}\), the associated self-consistent strong response is pathwise unique within the class covered by Assumption (H7)(ii).
\end{theorem}
\begin{proof}
See Appendix \ref{app:proof_strong_equilibrium}.
\end{proof}
\section{An Interbank Lending Example with WPR Observation}
\label{sec:lq_interbank}

This section illustrates Sections~2--5 through an interbank lending model with WPR observation.  The reserve dynamics follow \cite{carmona2015systemic}; a latent liquidity pressure shifts borrowing costs and is observed through a common signal.  We compare the resulting PR and WPR equilibria.

\subsection{Model and equilibrium equation systems}
\label{subsec:lq_closed_loop}

Let \(B\) and \(W\) be independent one-dimensional Brownian motions.  Let \(\Theta_0\sim N(\overline\Theta_0,\mathsf V_0^\Theta)\), with \(\mathsf V_0^\Theta>0\), be independent of them.  The latent pressure and the public observation satisfy
\begin{equation}
 d\Theta_t=-\varrho_\Theta\Theta_tdt,
 \qquad
 dy_t=h\Theta_tdt+\sigma_y dB_t,
 \qquad y_0=0,
 \label{eq:lq_factor_observation}
\end{equation}
where \(\varrho_\Theta>0\), \(h\ne0\), and \(\sigma_y>0\).  Perfect recall uses \(\mathcal F_t^Y=\sigma(y_s:0\le s\le t)\), with the usual augmentation, whereas WPR uses \(\mathcal G_t^I=\sigma(y_t)\).

For a frozen conditional mean \(m\) and a frozen mean control \(\overline u\), the reserve of the representative bank is
\begin{equation}
 dx_t=\bigl[a(m_t-x_t)+u_t\bigr]dt+\sigma dW_t,
 \qquad x_0\in L^2,
 \label{eq:lq_reserve}
\end{equation}
where \(a\ge0\), \(\sigma>0\), and \(x_0\) is independent of \((\Theta_0,B,W)\), with \(m_0=\mathbb E[x_0]\).  The constant \(c\in\mathbb R\) is a deterministic terminal target.  The cost is
\begin{equation}
\begin{aligned}
 J(u;m,\overline u)=\mathbb E\Bigg[\int_0^T\Big(&
 \frac{\mathsf R_u}{2}u_t^2+\mathsf R_cu_t\overline u_t-\Theta_tu_t
 -\mathsf q_{ux}u_t(m_t-x_t)
 +\frac{\mathsf Q_x}{2}(m_t-x_t)^2\Big)dt\\
 &+\frac{\mathsf Q_T}{2}(x_T-c)^2\Bigg].
\end{aligned}
\label{eq:lq_cost_new}
\end{equation}
We assume
\begin{equation}
 \mathsf R_u>0,\quad \mathsf R_c\ge0,\quad
 a,\mathsf q_{ux},\mathsf Q_x,\mathsf Q_T\ge0,\qquad
 \mathsf R_u\mathsf Q_x>\mathsf q_{ux}^{\,2}.
 \label{eq:lq_strict_convexity}
\end{equation}
Here \(\mathsf R_cu_t\overline u_t\) represents borrowing congestion, \(-\Theta_tu_t\) transmits latent liquidity pressure to marginal cost, and \(-\mathsf q_{ux}u_t(m_t-x_t)\) links borrowing to the reserve gap.  At equilibrium,
\begin{equation}
 m_t=\mathbb E[x_t\mid\mathcal F_t^Y],
 \qquad \overline u_t=\mathbb E[u_t\mid\mathcal F_t^Y]=u_t.
 \label{eq:lq_consistency_new}
\end{equation}
Set \(\mathsf R_{\rm eff}:=\mathsf R_u+\mathsf R_c\).

For an admissible control \(u\), define \(\chi_t^u:=\mathbb E[x_t^u\mid\mathcal F_t^Y]\).  Then
\begin{equation}
 d\chi_t^u=\bigl[a(m_t-\chi_t^u)+u_t\bigr]dt,
 \qquad
 d(x_t^u-\chi_t^u)=-a(x_t^u-\chi_t^u)dt+\sigma dW_t.
 \label{eq:lq_projection}
\end{equation}
The second equation is independent of \(u\), so conditional variance terms do not affect the minimizer.

Let \(\mathcal U_{\rm PR}^{\rm LQ}\) denote the square-integrable \(\mathbb F^Y\)-progressively measurable controls.  The WPR space \(\mathcal U_{\rm WPR}^{\rm LQ}\) consists of the square-integrable controls \(u_t=\phi(t,y_t)\), where \(\phi\) is jointly Borel.  Both are closed Hilbert spaces under \(\|u\|_2^2:=\mathbb E\int_0^T|u_t|^2dt\), and the next result covers both information structures.

\begin{proposition}
\label{prop:lq_frozen_response}
Let \((m,\overline u)\) be an \(\mathbb F^Y\)-progressively measurable pair satisfying
\[
 \mathbb E\!\int_0^T(|m_t|^2+|\overline u_t|^2)dt<\infty.
\]
The projected problem has a unique minimizer in the PR control space and in the WPR control space.  Let \(\varpi^u\) be the corresponding adjoint process.  At equilibrium, the optimal control satisfies
\begin{equation}
\begin{cases}
 dm_t=u_tdt,\qquad m_0=\mathbb E[x_0],\\
 d\varpi_t^u=(a\varpi_t^u-\mathsf q_{ux}u_t)dt+d\mathcal N_t^u,\\
 \varpi_T^u=\mathsf Q_T(m_T-c),\\
 \mathsf R_{\rm eff}u_t
 =\mathbb E[\Theta_t-\varpi_t^u\mid\mathcal I_t],
\end{cases}
\label{eq:lq_common_optimality}
\end{equation}
Here \(\mathcal I_t=\mathcal F_t^Y\) under PR and \(\mathcal I_t=\mathcal G_t^I\) under WPR.  The process \(\mathcal N^u\) is a square-integrable \(\mathbb F^Y\)-martingale, normalized by \(\mathcal N_0^u=0\).  This system characterizes the unique global minimizer in the corresponding information class.
\end{proposition}
\begin{proof}
See Appendix~\ref{app:proof_lq_frozen_response}.
\end{proof}

\noindent\textbf{[PR Equation System].} Set \(\widehat\Theta_t:=\mathbb E[\Theta_t\mid\mathcal F_t^Y]\), and let \(\mathsf V_t^{\Theta\mid{\rm PR}}\) be the filtering error variance.  The Kalman--Bucy equations are
\begin{equation}
\begin{aligned}
 d\widehat\Theta_t&=-\varrho_\Theta\widehat\Theta_tdt
 +K_t^{\rm f}(dy_t-h\widehat\Theta_tdt),
 &K_t^{\rm f}&=\frac{h\mathsf V_t^{\Theta\mid{\rm PR}}}{\sigma_y^2},\\
 \dot{\mathsf V}_t^{\Theta\mid{\rm PR}}
 &=-2\varrho_\Theta\mathsf V_t^{\Theta\mid{\rm PR}}
 -\frac{h^2}{\sigma_y^2}
 (\mathsf V_t^{\Theta\mid{\rm PR}})^2,
 &\mathsf V_0^{\Theta\mid{\rm PR}}&=\mathsf V_0^\Theta.
\end{aligned}
\label{eq:lq_pr_filter_new}
\end{equation}
with \(\widehat\Theta_0=\overline\Theta_0\). The normalized innovation \(d\widehat B_t^Y=\sigma_y^{-1}(dy_t-h\widehat\Theta_tdt)\) is an \(\mathbb F^Y\)-Brownian motion.  Define the deterministic coefficients \(\mathfrak r\) and \(\mathfrak s\) by
\begin{equation}
\begin{aligned}
 \dot{\mathfrak r}_t
 &=\frac{\mathfrak r_t^2}{\mathsf R_{\rm eff}}
 +\left(a+\frac{\mathsf q_{ux}}{\mathsf R_{\rm eff}}\right)
 \mathfrak r_t,
 &\mathfrak r_T&=\mathsf Q_T,\\
 \dot{\mathfrak s}_t
 &=\left(a+\varrho_\Theta
 +\frac{\mathsf q_{ux}+\mathfrak r_t}{\mathsf R_{\rm eff}}\right)
 \mathfrak s_t
 -\frac{\mathsf q_{ux}+\mathfrak r_t}{\mathsf R_{\rm eff}},
 &\mathfrak s_T&=0.
\end{aligned}
\label{eq:lq_pr_riccati_new}
\end{equation}
The PR feedback is
\begin{equation}
 u_t^{\rm PR}
 =-\frac{\mathfrak r_t}{\mathsf R_{\rm eff}}(m_t^{\rm PR}-c)
 +\frac{1-\mathfrak s_t}{\mathsf R_{\rm eff}}\widehat\Theta_t.
 \label{eq:lq_pr_control_new}
\end{equation}
Thus the PR system consists of \eqref{eq:lq_factor_observation}, \eqref{eq:lq_pr_filter_new}, \eqref{eq:lq_pr_riccati_new}, and
\begin{equation}
\begin{cases}
 dm_t^{\rm PR}=u_t^{\rm PR}dt,
 &m_0^{\rm PR}=m_0,\\
 dx_t^{\rm PR}
 =[a(m_t^{\rm PR}-x_t^{\rm PR})+u_t^{\rm PR}]dt+\sigma dW_t.
\end{cases}
\label{eq:lq_pr_closed_new}
\end{equation}

\noindent\textbf{[WPR Equation System].} Under WPR, \eqref{eq:lq_common_optimality} and the linear adjoint equation give
\begin{equation}
\begin{aligned}
\mathsf R_{\rm eff}u_t
&+e^{-a(T-t)}\mathsf Q_T\int_0^T\mathbb E[u_s\mid y_t]ds
+\mathsf q_{ux}\int_t^Te^{-a(s-t)}\mathbb E[u_s\mid y_t]ds\\
&=\mathbb E[\Theta_t\mid y_t]
-e^{-a(T-t)}\mathsf Q_T(m_0-c).
\end{aligned}
\label{eq:lq_wpr_exact_new}
\end{equation}
We first solve this equation over the full closed space of square-integrable controls of the form \(u_t=\phi(t,y_t)\); no affine restriction is imposed.

The Gaussian structure then yields a finite-dimensional representation.  Write \(\overline\Theta_t=\mathbb E[\Theta_t]\) and \(\overline y_t=\mathbb E[y_t]\), and define
\begin{equation}
 \mathsf V_t^\Theta=\operatorname{Var}(\Theta_t),\quad
 \mathsf C_t^{\Theta y}=\operatorname{Cov}(\Theta_t,y_t),\quad
 \mathsf V_t^y=\operatorname{Var}(y_t),\quad
 \mathsf C^{yy}(s,t)=\operatorname{Cov}(y_s,y_t).
 \label{eq:lq_covariance_notation_new}
\end{equation}
These quantities satisfy
\begin{equation}
\begin{aligned}
 \dot{\overline\Theta}_t&=-\varrho_\Theta\overline\Theta_t,
 &\dot{\overline y}_t&=h\overline\Theta_t,\\
 \dot{\mathsf V}_t^\Theta&=-2\varrho_\Theta\mathsf V_t^\Theta,
 &\dot{\mathsf C}_t^{\Theta y}
 &=-\varrho_\Theta\mathsf C_t^{\Theta y}+h\mathsf V_t^\Theta,\\
 \dot{\mathsf V}_t^y&=2h\mathsf C_t^{\Theta y}+\sigma_y^2.
\end{aligned}
\label{eq:lq_moment_odes_new}
\end{equation}
The initial values are \(\overline y_0=0\), \(\mathsf C_0^{\Theta y}=0\), and \(\mathsf V_0^y=0\), together with the prescribed \(\overline\Theta_0\) and \(\mathsf V_0^\Theta\).  For \(s\ge t\),
\begin{equation}
 \mathsf C^{yy}(s,t)=\mathsf V_t^y
 +\frac{h\mathsf C_t^{\Theta y}}{\varrho_\Theta}
 (1-e^{-\varrho_\Theta(s-t)}),
 \qquad \mathsf C^{yy}(t,s)=\mathsf C^{yy}(s,t).
 \label{eq:lq_cross_covariance_new}
\end{equation}

The unique WPR control is affine in the current observation:
\begin{equation}
 u_t^{\rm WPR}=\mathsf b_t^{\rm W}
 +\mathsf k_t^{\rm W}(y_t-\overline y_t).
 \label{eq:lq_wpr_affine_new}
\end{equation}
Its deterministic coefficients solve
\begin{equation}
\begin{aligned}
\mathsf R_{\rm eff}\mathsf b_t^{\rm W}
&+e^{-a(T-t)}\mathsf Q_T\left(m_0-c
 +\int_0^T\mathsf b_s^{\rm W}ds\right)\\
&+\mathsf q_{ux}\int_t^Te^{-a(s-t)}\mathsf b_s^{\rm W}ds
=\overline\Theta_t,
\end{aligned}
\label{eq:lq_wpr_mean_new}
\end{equation}
and, for \(dt\)-a.e. \(t\in(0,T)\),
\begin{equation}
\begin{aligned}
\mathsf R_{\rm eff}\mathsf k_t^{\rm W}
&+\frac{e^{-a(T-t)}\mathsf Q_T}{\mathsf V_t^y}
 \int_0^T\mathsf k_s^{\rm W}\mathsf C^{yy}(s,t)ds\\
&+\frac{\mathsf q_{ux}}{\mathsf V_t^y}
 \int_t^Te^{-a(s-t)}\mathsf k_s^{\rm W}\mathsf C^{yy}(s,t)ds
=\frac{\mathsf C_t^{\Theta y}}{\mathsf V_t^y}.
\end{aligned}
\label{eq:lq_wpr_gain_new}
\end{equation}
The value of \(\mathsf k_0^{\rm W}\) is immaterial for the \(dt\otimes d\mathbb P\) equivalence class.

\begin{theorem}
\label{thm:lq_pr_wpr_new}
Under the preceding parameter assumptions, in particular \eqref{eq:lq_strict_convexity}, the PR system \eqref{eq:lq_pr_filter_new}--\eqref{eq:lq_pr_closed_new} has a unique strong solution.  Equations \eqref{eq:lq_wpr_mean_new}--\eqref{eq:lq_wpr_gain_new} have a unique solution
\[
 (\mathsf b^{\rm W},\mathsf k^{\rm W})
 \in L^2(0,T)\times
 L^2((0,T),e^{-at}\mathsf V_t^y dt).
\]
The associated WPR equilibrium is the unique strong solution of
\begin{equation}
\begin{cases}
 d\Theta_t=-\varrho_\Theta\Theta_tdt,
 &\Theta_0\sim N(\overline\Theta_0,\mathsf V_0^\Theta),\\
 dy_t=h\Theta_tdt+\sigma_y dB_t,\\
 u_t^{\rm WPR}=\mathsf b_t^{\rm W}
 +\mathsf k_t^{\rm W}(y_t-\overline y_t),\\
 dm_t^{\rm WPR}=u_t^{\rm WPR}dt,
 &m_0^{\rm WPR}=m_0,\\
 dx_t^{\rm WPR}
 =[a(m_t^{\rm WPR}-x_t^{\rm WPR})+u_t^{\rm WPR}]dt+\sigma dW_t.
\end{cases}
\label{eq:lq_wpr_closed_new}
\end{equation}
Both systems satisfy \eqref{eq:lq_consistency_new}.  Their controls are the unique best responses in the corresponding information classes.
\end{theorem}
\begin{proof}
See Appendix~\ref{app:proof_lq_pr_wpr_new}.
\end{proof}

Put \(g_t=he^{-\varrho_\Theta t}\).  The posterior variances of the centered initial factor under PR and WPR are
\begin{equation}
 \mathsf v_{\Theta_0}^{\rm PR}(t)=\left((\mathsf V_0^\Theta)^{-1}
 +\frac{\int_0^t g_s^2ds}{\sigma_y^2}\right)^{-1},
 \qquad
 \mathsf v_{\Theta_0}^{\rm WPR}(t)=\left((\mathsf V_0^\Theta)^{-1}
 +\frac{(\int_0^t g_sds)^2}{\sigma_y^2t}\right)^{-1}.
 \label{eq:lq_information_gap_new}
\end{equation}
Strict Cauchy--Schwarz gives \(\mathsf v_{\Theta_0}^{\rm PR}(t)<\mathsf v_{\Theta_0}^{\rm WPR}(t)\) for every \(t>0\), so the resulting controls and conditional mean paths may differ.

\begin{remark}[Relation to the general model]
This example retains the WPR information structure of Sections~2--5, with the path posterior and extended WPR belief measure reduced to Gaussian conditional moments.  The mean-control term and unbounded action space are handled directly by the LQ first-order condition in square-integrable Hilbert spaces.
\end{remark}

\subsection{Numerical experiment and discussion}
\label{subsec:lq_numerics}

We next compare the equilibrium behavior under PR and WPR numerically.  The parameters are
\begin{equation*}
\begin{gathered}
T=4,\ a=0.8,\ \mathsf q_{ux}=0.15,\quad
\mathsf R_u=0.35,\ \mathsf R_c=0.05,\quad
\mathsf Q_x=0.8,\ \mathsf Q_T=1,\\
\varrho_\Theta=0.8,\ h=1,\ \sigma_y=1.8,\ \sigma=0.4,\quad
\overline\Theta_0=1.2,\ \mathsf V_0^\Theta=20,\ c=0,\\
x_0\sim N(-0.8,0.2),\qquad m_0=-0.8.
\end{gathered}
\end{equation*}
Here \(\mathsf R_u\mathsf Q_x=0.28>\mathsf q_{ux}^{\,2}=0.0225\).  The cost summary uses 8000 coupled Monte Carlo replications.

\begin{figure}[!htbp]
 \centering
 \begin{minipage}{.38\textwidth}
  \centering
  \includegraphics[width=\linewidth]{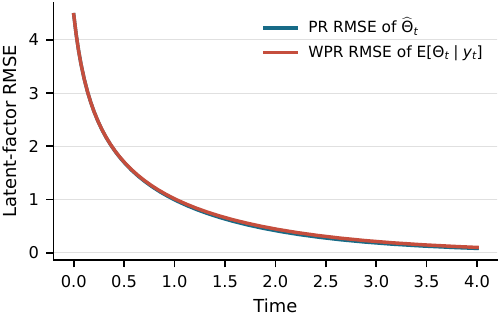}
 \end{minipage}\hfill
 \begin{minipage}{.38\textwidth}
  \centering
  \includegraphics[width=\linewidth]{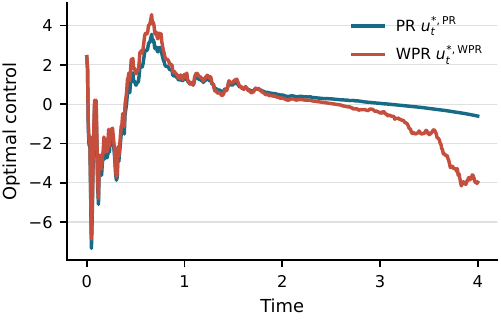}
 \end{minipage}
 \caption{Information loss and the feedback response.  Left: root-mean-square errors of the PR and WPR estimators.  Right: one coupled realization of \(u_t^{*,{\rm PR}}\) and \(u_t^{*,{\rm WPR}}\).}
 \label{fig:lq_estimation_control_new}
\end{figure}
\FloatBarrier

PR uses the observation history and has slightly smaller hidden-factor error (integrated RMSEs \(1.131\) and \(1.140\)).  The distinct feedback laws amplify this gap: the coupled controls have root-mean-square difference \(1.066\).

\begin{figure}[!htbp]
 \centering
 \begin{minipage}{.34\textwidth}
  \centering
  \includegraphics[width=\linewidth]{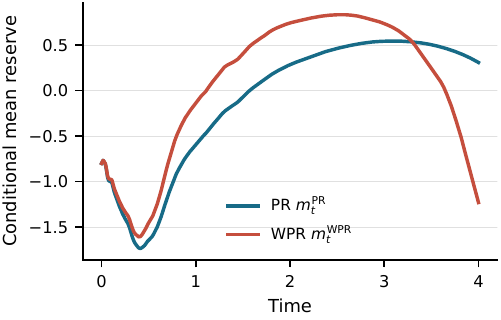}
 \end{minipage}\hfill
 \begin{minipage}{.34\textwidth}
  \centering
  \includegraphics[width=\linewidth]{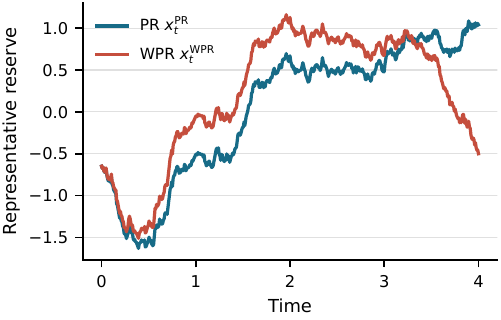}
 \end{minipage}
 \caption{Transmission from the feedback gap to reserves.  Left: conditional mean reserves.  Right: representative reserves under the same private-noise realization.}
 \label{fig:lq_mean_state_new}
\end{figure}
\FloatBarrier

The feedback gap accumulates in the conditional mean through \(dm_t^I=u_t^I dt\).  Under the coupled initial reserve and private noise, the deviations from the respective conditional means, \(x_t^{\rm PR}-m_t^{\rm PR}\) and \(x_t^{\rm WPR}-m_t^{\rm WPR}\), coincide; the uncentered state paths remain separated by the conditional-mean gap.

\begin{figure}[!htbp]
 \centering
 \includegraphics[width=.38\textwidth]{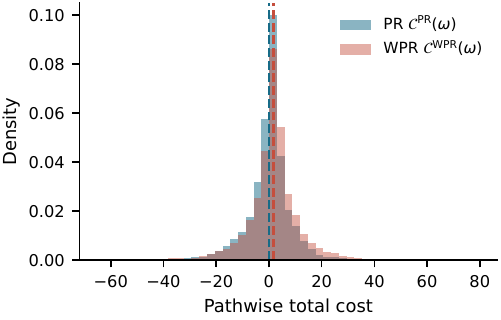}
 \caption{Monte Carlo distributions of the pathwise equilibrium costs \(\mathcal C^{\rm PR}(\omega)\) and \(\mathcal C^{\rm WPR}(\omega)\); dashed lines mark the sample means.}
 \label{fig:lq_cost_new}
\end{figure}
\FloatBarrier

The estimated mean costs are \(-0.031\) under PR and \(1.785\) under WPR. WPR uses only the current observation and thus lacks the temporal information with which PR responds to latent liquidity pressure and associated reserve deviations, raising its average cost.

\section{Conclusion}

This paper studied partially observed mean field games without perfect recall. Control uses the generally non-nested family \(\mathcal G_t^I=\sigma(y_t)\), while the conditional population law follows the observation filtration \(\mathbb F^Y\). Unequal histories impede joint equilibrium closure. We encode the environment by a deterministic compatible joint law \(P\) of the state, random mean field term, and driving variables, retaining dependence without enlarging the agent's control information.

For fixed \(P\), compactness of the strategy class and continuity of the likelihood-weighted cost yield an optimizer; Girsanov's theorem gives the WPR maximum principle, and the WPR belief measure defines a strongly convex conditional Hamiltonian with a Lipschitz selector. The joint path posterior supplies the weak Kushner-Stratonovich coefficients. Uniform moment and time-modulus bounds, lower semicontinuity, and closure under mixtures give a compact convex law set; selector and filter stability make the recursive response continuous. Schauder-Tychonoff then yields a weak WPR equilibrium. For fixed equilibrium law and feedback, pathwise uniqueness and the compatible Yamada-Watanabe theorem give a strong realization.

The LQ interbank model illustrates these information effects. PR gives Kalman-Bucy feedback, whereas WPR leads to a Fredholm-Volterra equation whose Gaussian solution is affine in current observation. Numerics compare estimation, feedback, and cost. However, strong realization requires strong stability estimates. Outside the Gaussian setting, the joint path posterior is generally difficult to close in finite dimensions, which complicates computation. Future work will weaken these stability conditions, study finite-agent approximations, and develop tractable schemes for nonlinear WPR models. We also plan to move beyond the partially observed state setting and examine WPR under more general non-nested information structures.

\appendix

\section{Proofs for the Reference-Measure Formulation}
\label{app:reference_measure_proofs}

\subsection{Proof of Theorem \ref{thm:ref_solution}}
\label{app:proof_ref_solution}

\begin{proof}
Assumption (H2) gives a progressively measurable, bounded, and pathwise Lipschitz diffusion coefficient, so the standard successive-approximation argument gives a pathwise unique strong solution; see \cite[Theorem 16.3.11]{cohen2015stochastic}.  Put \(d=d_x+d_y\), \(K_\Sigma=\sup_{t,X}\|\Sigma(t,X)\|\), and \(M=\int_0^\cdot\Sigma(s,X)dV_s\).  The BDG inequality gives the moment bound below.  Since \(\langle M^i\rangle_T\le K_\Sigma^2T\), the Dambis--Dubins--Schwarz theorem \cite[Chapter V]{revuz1999continuous} and the reflection principle \cite[Chapter III]{revuz1999continuous} give the tail bound, and hence
\[
\begin{gathered}
    \mathbb E^{\mathbb P}[\|X\|_T^{2p}]
    \le C_p\left(1+\mathbb E^{\mathbb P}[|X_0|^{2p}]\right),
    \qquad p\ge1,\\
    \mathbb P(\|M\|_T>r)
    \le 4d\exp\!\left(-\frac{r^2}{2dK_\Sigma^2T}\right),\qquad r>0,\\
    \mathbb E^{\mathbb P}[e^{\beta\|M\|_T^2}]<\infty
    \quad\text{for }\beta<(2dK_\Sigma^2T)^{-1},\\
    \mathbb E^{\mathbb P}[e^{\alpha\|X\|_T^2}]
    \le
    \mathbb E^{\mathbb P}[e^{4\alpha|X_0|^2}]^{1/2}
    \mathbb E^{\mathbb P}[e^{4\alpha\|M\|_T^2}]^{1/2}<\infty
\end{gathered}
\]
when \(4\alpha<\min\{\alpha_0,(2dK_\Sigma^2T)^{-1}\}\); the last line uses \(\|X\|_T^2\le2|X_0|^2+2\|M\|_T^2\) and H\"older's inequality.  The first line is finite by Assumption (H1)(ii), and taking \(\alpha=\alpha_X\) proves \eqref{eq:reference_exponential_moment}.
\end{proof}

\subsection{Proof of Theorem \ref{thm:girsanov}}
\label{app:proof_girsanov}

\begin{proof}
Let \(K_U=\sup_{a\in\mathbb U}|a|\), \(K_{\Sigma^{-1}}=\|\Sigma^{-1}\|_\infty\), and \(C_\mu=\sup_{t\le T}\mathcal W_2(\mu_t,\delta_0)\).  Assumptions (H2)--(H3) give
\begin{equation}
\begin{gathered}
    |\theta(t)|\le a_\mu+b\|X\|_T,\qquad
    a_\mu:=K_{\Sigma^{-1}}
    \big(C_{F_0}(1+C_\mu)+C_{F_1}K_U\big),
    \quad b:=K_{\Sigma^{-1}}C_{F_0},\\
    \frac12\int_0^T|\theta(t)|^2dt
    \le Ta_\mu^2+Tb^2\|X\|_T^2.
\end{gathered}
    \label{eq:appendix_theta_bound}
\end{equation}
Since \(Tb^2=C_4(T)<\alpha_X\), \eqref{eq:reference_exponential_moment} verifies Novikov's condition. Girsanov's theorem \cite{karatzas1991brownian} then makes \(V^u\) Brownian under \(\mathbb P^u\), and substitution gives \eqref{eq:augmented_dyn}.

For weak uniqueness, fix the feedback strategy \(\phi\) and let \((\bar X,\bar V,\bar{\mathbb P})\) be another weak solution with the same initial law and fixed flow, with \(\bar u_t=\phi(t,\bar y_t)\).  Write \(\bar\theta_t=\Sigma^{-1}(t,\bar X) F(t,\bar X,\bar u_t,\mu_t)\), and set
\begin{align*}
    \tau_n&:=\inf\left\{t:\int_0^t|\bar\theta_s|^2ds\ge n\right\}\wedge T,\\
    Z_t^n&:=
    \mathcal E\!\left(-\int_0^{\cdot\wedge\tau_n}
    \bar\theta_s^\top d\bar V_s\right)_t.
\end{align*}
The stopped exponentials are martingales.  Define \(d\mathbb Q_n=Z_T^n\,d\bar{\mathbb P}\).  Since \(\mathbb E^{\bar{\mathbb P}}[Z_T^n\mid\mathcal F_0]=1\), this change of measure preserves the initial law.  Under \(\mathbb Q_n\), \(\widetilde V^n\) below is Brownian, and BDG, boundedness of \(\Sigma\), and the stopped growth estimate give
\[
\begin{gathered}
    \widetilde V_t^n
    :=\bar V_t+\int_0^{t\wedge\tau_n}\bar\theta_sds,\qquad
    \bar X_{t\wedge\tau_n}
    =\bar X_0+\int_0^{t\wedge\tau_n}
      \Sigma(s,\bar X)d\widetilde V_s^n,\\
    \sup_n\mathbb E^{\mathbb Q_n}
    [\|\bar X\|_{\tau_n}^2]
    \le C\left(1+\mathbb E[|X_0|^2]\right)<\infty,\\
    \sup_n\mathbb E^{\bar{\mathbb P}}[Z_T^n\log Z_T^n]
    =\frac12\sup_n\mathbb E^{\mathbb Q_n}\!\left[
      \int_0^{\tau_n}|\bar\theta_s|^2ds\right]<\infty.
\end{gathered}
\]
Thus \(Z_T^n\to Z_T\) almost surely; the entropy bound makes the sequence uniformly integrable.  The convergence therefore holds in \(L^1\), with \(\mathbb E^{\bar{\mathbb P}}[Z_T]=1\).  Under \(d\mathbb Q=Z_Td\bar{\mathbb P}\), \(\widetilde V\) below is Brownian and \(\bar X\) solves the reference equation, whose law is unique by \cref{thm:ref_solution}.  Moreover,
\[
\begin{gathered}
    \widetilde V_t=\bar V_t+\int_0^t\bar\theta_sds,\\
    \left.\frac{d\bar{\mathbb P}}{d\mathbb Q}
    \right|_{\mathcal F_T}
    =\mathcal E\!\left(\int_0^\cdot
      \bar\theta_s^\top d\widetilde V_s\right)_T.
\end{gathered}
\]
The inverse-density identity follows from Girsanov.  Since \(\Sigma\) is nonsingular, the reference equation yields
\[
    \widetilde V_t=\int_0^t\Sigma^{-1}(s,\bar X)d\bar X_s.
\]
Hence, for fixed \(\mu\) and \(\phi\), the integrand \(\bar\theta\) and inverse density \(L_T:=d\bar{\mathbb P}/d\mathbb Q\) are measurable functionals of the reference state path \(\bar X\).  Inverse Girsanov weighting of the unique reference law therefore determines a unique controlled joint law.
\end{proof}

\section{Proofs for the Fixed-P Variational System}
\label{app:wpr_smp_proofs}

\subsection{Proof of Lemma \ref{lem:fixed_P_controlled_moment}}
\label{app:proof_fixed_P_controlled_moment}

\begin{proof}
By \eqref{eq:environment_preservation}, the \(p_0\)-moments of \((X_0^P,\eta)\) are independent of \(\phi\) and finite.  Compactness of \(\mathbb U\) and Assumption (H3) give \( |F(t,X^P,\phi(t,y_t^P),\eta_t)| \le C(1+\|X^P\|_t+\mathcal W_2(\eta_t,\delta_0)) \) uniformly in \(\phi\).  Since \(\mathcal W_2^{p_0}(\eta_t,\delta_0)\le M_{p_0}(\eta_t)\), the BDG inequality and boundedness of \(\Sigma\) yield, for \(t\le T\),
\[
    \mathbb E^{\mathbb P^{P,\phi}}[\|X^P\|_t^{p_0}]
    \le C\!\left(1+\mathbb E^{\mathbb P^{P,\phi}}\!\left[
    |X_0^P|^{p_0}+\sup_{s\le T}M_{p_0}(\eta_s)\right]
    +\int_0^t\mathbb E^{\mathbb P^{P,\phi}}[\|X^P\|_r^{p_0}]\,dr\right).
\]
Gronwall's inequality and \eqref{eq:fixed_P_input_moment} prove \eqref{eq:fixed_P_controlled_moment}.
\end{proof}

\subsection{Proof of Theorem \ref{thm:adjoint}}
\label{app:proof_adjoint}

\begin{proof}
Fix \(\psi\in\Phi_{ad}\), write \(\bar\phi=\bar\phi^P\), \(\Delta\phi_t=\psi(t,y_t^P)-\bar\phi(t,y_t^P)\), and set \(\phi^\varepsilon=(1-\varepsilon)\bar\phi+\varepsilon\psi\) for \(\varepsilon\in[0,1]\).  By convexity, \(\phi^\varepsilon\in\Phi_{ad}\). Under \(\mathbb P^{P,\bar\phi}\), set \( A_t:=\Sigma^{-1}(t,X^P)F_1(t,X^P,\eta_t)\Delta\phi_t \), which is bounded.  Then
\begin{equation}
\begin{gathered}
    Z_t^\varepsilon
    :=
    \left.
    \frac{d\mathbb P^{P,\phi^\varepsilon}}
         {d\mathbb P^{P,\bar\phi}}
    \right|_{\mathcal F_t^P}
    =
    \mathcal E\!\left(
    \varepsilon\int_0^\cdot A_t^\top
    dV_t^{P,\bar\phi}\right)_t,\\
    \sup_{0\le\varepsilon\le1}
    \mathbb E^{\mathbb P^{P,\bar\phi}}\!
    \left[\sup_{t\le T}|Z_t^\varepsilon|^4\right]<\infty,\qquad
    \mathbb E^{\mathbb P^{P,\bar\phi}}\!
    \left[\sup_{t\le T}|Z_t^\varepsilon-1|^4\right]
    \le C\varepsilon^4,\\
    \frac{Z_T^\varepsilon-1}{\varepsilon}
    \longrightarrow
    \int_0^T A_t^\top dV_t^{P,\bar\phi}
    \quad\text{in }L^2(\mathbb P^{P,\bar\phi}).
\end{gathered}
\label{eq:fixed_P_density_derivative}
\end{equation}
The last two lines follow from standard exponential-martingale estimates and the BDG inequality.

Set
\begin{equation}
\begin{gathered}
    C^\varepsilon
    :=
    \int_0^T
    l(t,X^P,\phi^\varepsilon(t,y_t^P),\eta_t)dt
    +g(X_T^P),\\
    \frac{C^\varepsilon-C^0}{\varepsilon}
    \longrightarrow
    \int_0^T
    \left\langle
    \nabla_ul(t,X^P,\bar\phi(t,y_t^P),\eta_t),
    \psi(t,y_t^P)-\bar\phi(t,y_t^P)
    \right\rangle dt
    \quad\text{in }L^2(\mathbb P^{P,\bar\phi}).
\end{gathered}
\label{eq:fixed_P_cost_derivative}
\end{equation}
Here \(C^0\in L^2\), and the convergence follows from (H0), (H4), \cref{lem:fixed_P_controlled_moment}, the mean-value formula, and dominated convergence.

Let \(M_t^P:=\mathbb E^{\mathbb P^{P,\bar\phi}} [C^0\mid\mathcal F_t^P]\). Its Kunita--Watanabe decomposition with respect to \(V^{P,\bar\phi}\) is
\begin{equation}
    M_t^P=M_0^P+\int_0^t(q_s^P)^\top
    dV_s^{P,\bar\phi}+N_t^P,
    \qquad
    \langle N^P,V^{P,\bar\phi,j}\rangle_t=0 .
    \label{eq:fixed_P_KW}
\end{equation}
Thus \(p_t^P:=M_t^P-\int_0^t l(s,X^P,\bar\phi(s,y_s^P),\eta_s)ds\) satisfies \eqref{eq:adjoint_bsde}, with the asserted square integrability.

Moreover, \(J_P(\phi^\varepsilon) =\mathbb E^{\mathbb P^{P,\bar\phi}}[Z_T^\varepsilon C^\varepsilon]\). The two \(L^2\) limits and strong orthogonality in \eqref{eq:fixed_P_KW} give
\[
    0\le\left.\frac{d}{d\varepsilon}J_P(\phi^\varepsilon)
    \right|_{\varepsilon=0+}
    =\mathbb E^{\mathbb P^{P,\bar\phi}}\!\int_0^T
    \left\langle
    \nabla_u\mathcal H
    (t,X^P,\bar\phi(t,y_t^P),\eta,q_t^P),
    \Delta\phi_t\right\rangle dt .
\]
By \eqref{eq:hamiltonian}, this is \eqref{eq:variational_ineq}.
\end{proof}

\subsection{Proof of Theorem \ref{thm:wpr_smp}}
\label{app:proof_wpr_variational}

\begin{proof}
Set \(\mathbb Q:=\mathbb P^{P,\bar\phi^P}\), \(Y_t:=y_t^P\), and
\[
 G(t,\omega,u):=\nabla_u\mathcal H
 (t,X^P(\omega),u,\eta(\omega),q_t^P(\omega)).
\]
Uniformly in \(u\), \( |G(t,\omega,u)| \le C(1+|q_t^P|+\|X^P\|_t+\mathcal W_2(\eta_t,\delta_0)) \), and this bound is integrable by \cref{lem:fixed_P_controlled_moment}. Disintegrate \(dt\otimes d\mathbb Q\) through \((t,\omega)\mapsto(t,Y_t(\omega))\).  On a countable dense subset of \(\mathbb U\), this gives jointly measurable conditional versions outside one common null set.  Continuity of \(G\) in \(u\) and conditional dominated convergence extend the same version to every \(u\in\mathbb U\).  It may therefore be evaluated at the random action \(\bar\phi^P(t,Y_t)\).  Since \(\psi(t,Y_t)-\bar\phi^P(t,Y_t)\) is \(\sigma(Y_t)\)-measurable, the tower property gives, for \(dt\)-almost every \(t\),
\begin{align*}
    &\mathbb E^{\mathbb P^{P,\bar\phi^P}}\!\left[
    \left\langle
    \nabla_u\mathcal H
    (t,X^P,\bar\phi^P(t,y_t^P),\eta,q_t^P),
    \psi(t,y_t^P)-\bar\phi^P(t,y_t^P)
    \right\rangle\right]\\
    &\quad=
    \mathbb E^{\mathbb P^{P,\bar\phi^P}}\!\left[
    \left\langle
    \nabla_u\widehat{\mathcal H}^{P}
    (t,y_t^P,\bar\phi^P(t,y_t^P)),
    \psi(t,y_t^P)-\bar\phi^P(t,y_t^P)
    \right\rangle\right].
\end{align*}
Integrating in time and using \eqref{eq:variational_ineq} proves \eqref{eq:smp_condition}.
\end{proof}

\section{Proofs for the Fixed-P Best Response}
\label{app:best_response_proofs}

\subsection{Proof of Lemma \ref{lem:pointwise_hamiltonian_gradient}}
\label{app:proof_pointwise_hamiltonian_gradient}

\begin{proof}
Put \( G_t(\xi,u):=\nabla_u\mathcal H (t,\xi^X,u,\xi^\mu,\xi^q) \). Assumptions (H0) and (H2)--(H4) give {\small
\[
\begin{gathered}
    D_t(\xi,\zeta)
      :=\|\xi^X-\zeta^X\|_t+
        \sup_{r\le t}\mathcal W_2(\xi_r^\mu,\zeta_r^\mu),\\
    A_t(\xi)
      :=F_1(t,\xi^X,\xi_t^\mu)^\top
        \Sigma^{-1}(t,\xi^X)^\top,\qquad
    |A_t(\xi)|\le C_A,\qquad
    |A_t(\xi)-A_t(\zeta)|\le L_A D_t(\xi,\zeta),\\
    |\nabla_ul(t,\xi^X,u,\xi_t^\mu)
      -\nabla_ul(t,\zeta^X,u,\zeta_t^\mu)|
      \le L_{\nabla l}D_t(\xi,\zeta).
\end{gathered}
\]
}
Thus, using \(G_t(\xi,u)=A_t(\xi)\xi^q+ \nabla_ul(t,\xi^X,u,\xi_t^\mu)\),
\[
\begin{gathered}
    |G_t(\xi,u)-G_t(\zeta,u)|
    \le C_A|\xi^q-\zeta^q|
      +(L_A|\zeta^q|+L_{\nabla l})D_t(\xi,\zeta),\\
    \le L_{\mathcal H}c_{\mathcal H,t}(\xi,\zeta).
\end{gathered}
\]
This proves \eqref{eq:pointwise_hamiltonian_gradient}.
\end{proof}

\subsection{Proof of Lemma \ref{lem:conditional_gradient_stability}}
\label{app:proof_conditional_gradient_stability}

\begin{proof}
The growth assumptions and (H6)(ii) ensure that \(\Psi^P(t,y,u)\) is finite.  Fix \(t,y,y',u\), and choose a coupling \(\Gamma_\varepsilon\) of \(\widehat\pi_{t|y}^{P,*}\) and \(\widehat\pi_{t|y'}^{P,*}\) within \(\varepsilon\) of the transport cost in \eqref{eq:hamiltonian_gradient_cost}.  By \cref{lem:pointwise_hamiltonian_gradient} and (H6)(i),
\begin{align*}
    |\Psi^P(t,y,u)-\Psi^P(t,y',u)|
    &\le L_{\mathcal H}
    \int c_{\mathcal H,t}(\xi,\zeta)
    \Gamma_\varepsilon(d\xi,d\zeta)\\
    &\le L_{\mathcal H}
    \bigl(C_\pi|y-y'|+\varepsilon\bigr).
\end{align*}
Letting \(\varepsilon\downarrow0\) and adding (H6)(iv) proves \eqref{eq:H6_observation_gradient_stability}.
\end{proof}

\subsection{Proof of Theorem \ref{thm:fixed_P_existence}}
\label{app:proof_fixed_P_existence}

\begin{proof}
\textit{Step 1: existence.} Let \(\{\phi^n\}_{n\ge1}\subset\Phi_{ad}\) be a minimizing sequence. Arzel\`a--Ascoli and a diagonal argument, using uniform boundedness and equicontinuity on compact sets, give a subsequence and \(\bar\phi^P\in\Phi_{ad}\) such that
\begin{equation}
\begin{gathered}
    \phi^n\longrightarrow\bar\phi^P
    \quad\text{locally uniformly},\\
    \sup_{t\le T}
    |\phi^n(t,y_t^P)-\bar\phi^P(t,y_t^P)|
    \longrightarrow0
    \quad\mathbb P^{P,\bar\phi^P}\text{-a.s.}
\end{gathered}
    \label{eq:fixed_P_strategy_convergence}
\end{equation}
The second line uses the compact range of each continuous path \(y^P\). Under \(\mathbb P^{P,\bar\phi^P}\), set
\begin{equation}
\begin{split}
    Z_t^n
    &:=
    \left.
    \frac{d\mathbb P^{P,\phi^n}}
         {d\mathbb P^{P,\bar\phi^P}}
    \right|_{\mathcal F_t^P}
    =
    \mathcal E\!\left(
    \int_0^\cdot(\Delta\theta_s^n)^\top
    dV_s^{P,\bar\phi^P}\right)_t,\\
    \Delta\theta_t^n
    &:=
    \Sigma^{-1}(t,X^P)F_1(t,X^P,\eta_t)
    [\phi^n(t,y_t^P)-\bar\phi^P(t,y_t^P)] .
\end{split}
    \label{eq:fixed_P_relative_likelihood}
\end{equation}
The kernels \(\Delta\theta^n\) are uniformly bounded and converge to zero in every finite \(L^r(dt\otimes d\mathbb P^{P,\bar\phi^P})\). Exponential-martingale estimates, BDG, and Cauchy--Schwarz give
\begin{equation}
\begin{gathered}
    \sup_n\mathbb E^{\mathbb P^{P,\bar\phi^P}}\!\left[
      \sup_{t\le T}|Z_t^n|^4\right]<\infty,\\
    \mathbb E^{\mathbb P^{P,\bar\phi^P}}\!\left[
      \sup_{t\le T}|Z_t^n-1|^2\right]\longrightarrow0 .
\end{gathered}
    \label{eq:fixed_P_relative_likelihood_convergence}
\end{equation}

Set \(C^\phi:=\int_0^T l(t,X^P,\phi(t,y_t^P),\eta_t)dt+g(X_T^P)\). The quadratic cost bound, the mean-value formula, \eqref{eq:fixed_P_strategy_convergence}, and \cref{lem:fixed_P_controlled_moment} give the following estimates, with all norms taken under \(\mathbb P^{P,\bar\phi^P}\):
\[
\begin{gathered}
    \sup_n\|C^{\phi^n}\|_{L^2(\mathbb P^{P,\bar\phi^P})}<\infty,
    \qquad
    C^{\phi^n}\longrightarrow C^{\bar\phi^P}
    \quad\text{in }L^2(\mathbb P^{P,\bar\phi^P}),\\
    |J_P(\phi^n)-J_P(\bar\phi^P)|
    \le \|Z_T^n-1\|_2\|C^{\phi^n}\|_2
      +\|C^{\phi^n}-C^{\bar\phi^P}\|_1\longrightarrow0.
\end{gathered}
\]
The last line uses \(J_P(\phi^n)=\mathbb E^{\mathbb P^{P,\bar\phi^P}} [Z_T^nC^{\phi^n}]\) and Cauchy--Schwarz, and proves \eqref{eq:fixed_P_optimizer}.

\textit{Step 2: pointwise feedback.} Under (H6)(i), (ii), and (iv), the selected extended WPR belief measure and (H0) make \(\widehat{\mathcal H}^{P}\) jointly Borel, finite, and continuous in \(u\).  Compactness of \(\mathbb U\) and the measurable maximum theorem give a Borel minimizer.  The affine drift and \(\lambda\)-strong convexity of \(l\) make it unique; denote it by \(\phi^P(t,y)\).  For \(a=\phi^P(t,y)\) and \(b=\phi^P(s,y')\), the two variational inequalities and strong monotonicity of \(\Psi^P(t,y,\cdot)\) give
\begin{equation}
\begin{gathered}
    \lambda|a-b|^2
    \le\left\langle
    \Psi^P(s,y',b)-\Psi^P(t,y,b),a-b
    \right\rangle,\\
    |\phi^P(t,y)-\phi^P(s,y')|
    \le\frac{L_\Psi}{\lambda}
    (|t-s|+|y-y'|).
\end{gathered}
    \label{eq:fixed_P_selector_estimate}
\end{equation}
The second line follows from \cref{lem:conditional_gradient_stability}.  Hence \(\phi^P\in\Phi_{ad}\) because \(L_\Phi\ge L_\Psi/\lambda\), and \cref{thm:lipschitz} follows.

\textit{Step 3: identification of the selected optimal response.} Use \(\psi=\phi^P\) in \eqref{eq:smp_condition}.  Pointwise optimality of \(\phi^P\) and strong monotonicity give, with \(\delta_t:=\phi^P(t,y_t^P)-\bar\phi^P(t,y_t^P)\),
\[
\begin{gathered}
    \left\langle
    \Psi^P(t,y_t^P,\bar\phi^P(t,y_t^P)),
    \delta_t
    \right\rangle
    \le-\lambda|\delta_t|^2,\\
    \delta_t=0
    \quad dt\otimes d\mathbb P^{P,\bar\phi^P}\text{-a.e.}
\end{gathered}
\]
where the second line follows by integrating the first and applying \eqref{eq:smp_condition}. The relative Girsanov kernel therefore vanishes, so \(\mathbb P^{P,\phi^P}=\mathbb P^{P,\bar\phi^P}\), and the two strategies have the same optimal cost.  This identifies the synthesized selector with the selected optimizer.
\end{proof}

\section{Proofs for the Conditional Mean Field Law}
\label{app:conditional_mean_field_proofs}

\subsection{Proof of Lemma \ref{lem:joint_posterior_innovation}}
\label{app:proof_joint_posterior_innovation}

\begin{proof}
\textit{Step 1: one posterior process.} On the Polish space \(\mathsf E\), the moment bound supplies an integrable inf-compact function \(H\ge1+\|\xi\|_T^2+\sup_{s\le T}M_2(\zeta_s)\).  Choose a countable convergence-determining \(\mathcal D\subset C_b(\mathsf E)\) and a dense countable time set.  Ji\v{r}ina's theorem \cite{jirina1959regular,kallenberg2002foundations} gives conditional kernels at these times.  C\`adl\`ag regularization of \(\mathbb E^{\mathbb P^*}[g(Z)\mid\mathcal F_t^Y]\), \(g\in\mathcal D\), together with conditional tightness from \(H\), yields a unique weakly c\`adl\`ag optional kernel \(\Pi_t\).  It is the conditional law of \(Z\) given \(\mathcal F_t^Y\), and Portmanteau gives \(\Pi_t(H)<\infty\), hence \(m_t^Y=\Pi_t\circ e_t^{-1}\in\mathcal P_2\). For \(f\) in the statement, conditional Fubini gives
\[
 \mathbb E^{\mathbb P^*}\int_0^T\Pi_t(|f_t|)dt
 =\mathbb E^{\mathbb P^*}\int_0^T|f_t(Z)|dt<\infty.
\]
Equality of the optional and predictable \(dt\otimes d\mathbb P^*\)-completions gives the predictable version.

\textit{Step 2: predictable representation under the controlled probability measure.} Under \(\mathbb P_0^P\), \(B_t^P=\int_0^t\sigma_2^{-1}(s,y)dy_s\) generates \(\mathbb F^Y\), so this filtration has the \(B^P\)-representation property.  Let \(\overline\Lambda_t:=\mathbb E^{\mathbb P_0^P} [\Lambda_T^{P,\phi^P}\mid\mathcal F_t^Y]\) and \(\theta_t^B:=\sigma_2^{-1}(t,y)h_t^*(x^*,\eta)\).  After localization, the moment bound, optional-projection covariation identity, and Bayes' formula give
\[
    d\overline\Lambda_t
    =\overline\Lambda_t\beta_t^\top dB_t^P,
    \qquad
    \beta_t
    :=\mathbb E^{\mathbb P^*}[\theta_t^B\mid\mathcal F_{t-}^Y]
    =\sigma_2^{-1}(t,y)\widehat h_t .
\]
Since \(\mathcal F_{t-}^Y=\mathcal F_t^Y\) for \(t>0\), Girsanov makes \(\overline B_t:=B_t^P-\int_0^t\beta_sds\) an \(\mathbb F^Y\)-Brownian motion under \(\mathbb P^*\).  The quotient formula for \(M=(\overline\Lambda M)/\overline\Lambda\) transfers the representation property to \(\overline B\).

Finally,
\[
    \nu_t=\int_0^tO_s\,d\overline B_s,
    \qquad
    O_s:=a_s^{-1/2}\sigma_2(s,y),
\]
where \(O_s\) is predictable and \(O_sO_s^\top=I\).  Thus the transformation preserves both Brownian motion and predictable representation.
\end{proof}

\section{Proofs for the WPR-MFG Equilibrium Results}
\label{app:fixed_point_estimates}

\subsection{Proof of Lemma \ref{lem:closed_loop_input_stability}}
\label{app:proof_closed_loop_input_stability}

\begin{proof}
Fix \(\phi\in\Phi_{ad}\), set \(K_{\mathbb U}:=\sup_{u\in\mathbb U}|u|\), and write
\[
 \overline F_\phi(t,X,\eta)
 :=F\bigl(t,X,\phi(t,y_t),\eta_t\bigr).
\]
Assumptions (H1)(i), (H3), and (H5) give, uniformly in \(\phi\),
\begin{equation}
\begin{aligned}
 |\overline F_\phi(t,X,\eta)-\overline F_\phi(t,X',\eta')|
 &\le C\{\|X-X'\|_t+\mathcal W_2(\eta_t,\eta_t')\},\\
 |\overline F_\phi(t,X,\eta)|
 &\le C\{1+\|X\|_t+\mathcal W_2(\eta_t,\delta_0)\}.
\end{aligned}
\label{eq:closed_loop_drift_lipschitz}
\end{equation}

Compatibility keeps \(V\) Brownian after adjoining \(\eta\).  The preceding bounds and moment assumptions give a pathwise unique continuous strong solution by successive approximation.

For \(\gamma=\mathcal L(X_0,V,\eta)\), freeze the coefficients on nested partitions with vanishing mesh.  The resulting Euler maps \(S_\phi^m:\mathcal E_{\rm in}\to\mathcal X\) are Borel and nonanticipative; BDG, Gronwall, continuity, and localization give, on every compatible realization of \(\gamma\),
\[
 \mathbb E\|S_\phi^m(X_0,V,\eta)-X^{\phi,\eta}\|_T^2\longrightarrow0.
\]
Choose an \(L^2(\gamma;\mathcal X)\)-summable subsequence.  Tonelli and extension off its full-measure convergence set yield a Borel nonanticipative limit \(S_{\phi,\gamma}\).  Pathwise uniqueness then gives
\[
 X^{\phi,\eta}=S_{\phi,\gamma}(X_0,V,\eta)\qquad\text{a.s.}
\]
Thus \(\mathcal L(X^{\phi,\eta},X_0,V,\eta) =(S_{\phi,\gamma},\operatorname{id})_\#\gamma\), proving uniqueness of the complete solution law.

For jointly compatible \(\eta,\eta'\) with the same \((X_0,V)\), write \(\Delta X=X^{\phi,\eta}-X^{\phi,\eta'}\).  Joint compatibility keeps \(V\) Brownian in the enlarged filtration; hence \eqref{eq:closed_loop_drift_lipschitz}, Assumption (H2), and the BDG inequality give
\[
 \mathbb E\|\Delta X\|_t^2
 \le C_1\int_0^t\mathbb E\|\Delta X\|_s^2ds
 +C_2\int_0^t\mathbb E\!\left[
 \sup_{r\le s}\mathcal W_2^2(\eta_r,\eta_r')\right]ds .
\]
Gronwall proves \eqref{eq:closed_loop_input_stability}, uniformly over \(\{\phi^P:P\in\mathfrak W\}\).

\end{proof}

\subsection{Proof of Lemma \ref{lem:compatible_law_closed}}
\label{app:proof_compatible_law_closed}

\begin{proof}
The marginal condition defining \(\mathfrak K\) is affine and closed.  For a bounded continuous cylinder function \(h\) of \((X_0,V)\), define
\[
 (K_th)(X_0,V_{\cdot\wedge t})
 :=\mathbb E^{\lambda_0\otimes\mathbb W_V}
 [h(X_0,V)\mid\mathcal F_t^{X_0,V}].
\]
The Wiener transition kernel makes \(K_th\) bounded and continuous. Compatibility is equivalent, for every bounded continuous \(\mathcal H_t\)-cylinder function \(G\), to
\begin{equation}
 \mathbb E^P\!\left[G\{h-K_th\}\right]=0.
 \label{eq:compatibility_cylinder_identity}
\end{equation}
Linearity proves convexity.  If \(P_n\Rightarrow P\), bounded continuity passes the identities to \(P\); a monotone-class argument recovers the defining conditional-expectation identity.  Thus \(\mathfrak K\) is closed.
\end{proof}

\subsection{Proof of Lemma \ref{lem:joint_response_wellposed}}
\label{app:proof_joint_response_wellposed}

\begin{proof}
Fix \(P\in\mathfrak W\).  Theorem \ref{thm:fixed_P_existence} and (H7)(i) give a deterministic \(\phi^P\in\Phi_{ad}\), so \(\phi^P(t,y_t)\) is \(\mathcal G_t^I\)-measurable.  On a compatible basis carrying \((X_0,V)\), impose conditional self-consistency recursively by
\[
 \eta_t^0:=\delta_0,\qquad X^n:=X^{P,\eta^n},\qquad
 \eta^{n+1}:=m^{P,\eta^n}.
\]
By (H7)(ii), these are nonanticipative functionals of \((X_0,V)\); compatibility is preserved after adjoining \((\eta^n,\eta^{n+1})\), so adjacent inputs are jointly compatible.

For the uniform moment bound, define
\[
 Z_n:=\mathbb E[\|X^n\|_T^{p_0}],\qquad
 S_n:=\mathbb E\!\left[\sup_{t\le T}M_{p_0}(\eta_t^n)\right].
\]
For any input \(\eta\), put \(N_t=\mathbb E^{\mathbb Q^{P,\eta}}[ \|X^{P,\eta}\|_T^{p_0}\mid\mathcal F_t^{y^{P,\eta}}]\).  Since \(M_{p_0}(m_t^{P,\eta})\le N_t\), conditional Jensen, Doob's inequality, H\"older's inequality, and \eqref{eq:H7_uniform_likelihood} give, for \(r>1\),
\begin{equation*}
 \mathbb E^{\mathbb Q^{P,\eta}}\!\left[\sup_{t\le T}M_{p_0}(m_t^{P,\eta})\right]
 \le \frac r{r-1}C_\Lambda^{1/(p_\Lambda r)}
 \left(\sup_{P,\eta}\mathbb E^{\mathbb P_0^{P,\eta}}
 \|X^{P,\eta}\|_T^{p_0rq_\Lambda}\right)^{1/(rq_\Lambda)}=:C_{\rm cond}.
\end{equation*}
The last term is uniformly finite by \eqref{eq:reference_exponential_moment}. Together with \eqref{eq:closed_loop_drift_lipschitz}, this yields
\[
 Z_n\le A_{p_0}^X(T)+c_{p_0}(T)S_n,\qquad
 S_{n+1}\le C_{\rm cond},\qquad
 \sup_n(Z_n+S_n)<\infty,
\]
where the last bound uses \(S_0=0\) and is uniform in \(P\).

For \(n\ge1\) and \(n\ge0\), respectively, set
\[
 D_n(t):=\mathbb E\|X^n-X^{n-1}\|_t^2,\qquad
 A_n(t):=\mathbb E\!\left[\sup_{s\le t}
 \mathcal W_2^2(\eta_s^{n+1},\eta_s^n)\right].
\]
Lemma \ref{lem:closed_loop_input_stability}, (H7)(ii), and repeated integration give, uniformly in \(P\),
\[
\begin{aligned}
 D_n(t)&\le C\int_0^tA_{n-1}(s)ds,\qquad
 A_n(t)\le L_{\rm fil}D_n(t),\\
 A_0(T)&\le C_{\rm cond}^{2/p_0},\qquad
 A_n(T)\le C_0\frac{(CT)^n}{n!}.
\end{aligned}
\]
Since \(\sum_nA_n(T)^{1/2}<\infty\), completeness gives limits \(\eta^n\to\mu^P\) and \(X^n\to X^P\) in the corresponding \(L^2\) path spaces.  An almost surely uniformly convergent subsequence, Fatou's lemma, and rational-time factorization preserve the moment bound and give adapted nonanticipative Borel versions.

Applying the state and filter maps once more to \(\mu^P\), stability gives
\[
 \mathbb E\|X^n-X^{P,\mu^P}\|_T^2
 +\mathbb E\!\left[\sup_{t\le T}
 \mathcal W_2^2(\eta_t^{n+1},m_t^{P,\mu^P})\right]\longrightarrow0.
\]
Hence \(X^P=X^{P,\mu^P}\), \(\mu^P=m^{P,\mu^P}\), and \(\mu_t^P=\mathcal L(x_t^P\mid\mathcal F_t^{y^P})\).  Set \(Q_n:=\mathcal L(X^n,V,\eta^{n+1})\) and write \(Q^P:=\mathcal L(X^P,V,\mu^P)\).  Each \(Q_n\) lies in \(\mathfrak K\), and joint \(L^2\)-convergence with \cref{lem:compatible_law_closed} gives \(Q^P\in\mathfrak K\).  The limiting equation and moment bound therefore make \(Q^P\) a compatible \(P\)-response.

For two compatible \(P\)-responses, the common-input randomization and joint-compatibility results of \cite{kurtz2014weak} put them on a jointly compatible coupling with the same \((X_0,V)\).  For \(D(t):=\mathbb E\|X^1-X^2\|_t^2\), state and filter stability give
\[
 D(t)\le C\int_0^t\!\left\{D(s)+
 \mathbb E\!\left[\sup_{r\le s}\mathcal W_2^2(\mu_r^1,\mu_r^2)\right]
 \right\}ds
 \le C(1+L_{\rm fil})\int_0^tD(s)ds.
\]
Gronwall gives \(X^1=X^2\) and \(\mu^1=\mu^2\).  Thus the strongly realized response law \(\mathcal R(P)=Q^P\) is well defined.
\end{proof}

\subsection{Proof of Lemma \ref{lem:W_compact}}
\label{app:proof_W_compact}

\begin{proof}
The law generated by \((X_0,V)\sim\lambda_0\otimes\mathbb W_V\), \(X_t=X_0\), and \(\mu_t=\delta_0\) belongs to \(\mathfrak W\): its moment is \(\mathbb E|X_0|^{p_0}\le K_1\) and both increments vanish.  Convexity follows from \cref{lem:compatible_law_closed} because both defining bounds are preserved under mixtures.

The \(V\)-marginal is fixed.  For the state coordinate and \(L>0\),
\[
 \sup_{P\in\mathfrak W}P(\|X\|_T>L)\le \frac{K_1}{L^{p_0}},\qquad
 \sup_{P\in\mathfrak W}\mathbb E^P|X_t-X_s|^{p_{\rm tm}}
 \le K_2|t-s|^{1+\alpha_{\rm tm}}.
\]
Thus the uniform Kolmogorov criterion gives tightness in \(\mathcal X\).

For \(\mathcal K_R:=\{\nu\in\mathcal P_2:M_{p_0}(\nu)\le R\}\), Markov's inequality and \(p_0>2\) give
\begin{gather*}
 \inf_{P\in\mathfrak W}P\bigl(\mu_\cdot\in C([0,T];\mathcal K_R)\bigr)
 \ge1-\frac{K_1}{R},\\
 \sup_{\nu\in\mathcal K_R}\nu(|\xi|>L)
 \le R L^{-p_0},\\
 \sup_{\nu\in\mathcal K_R}\int_{|\xi|>L}|\xi|^2\nu(d\xi)
 \le R L^{2-p_0}\xrightarrow[L\to\infty]{}0.
\end{gather*}
They give weak relative compactness and uniform integrability of second moments, while lower semicontinuity of \(M_{p_0}\) makes \(\mathcal K_R\) closed.  Hence \(\mathcal K_R\) is \(\mathcal W_2\)-compact.

Also \(R_\mu:=\sup_{r\le T}M_{p_0}(\mu_r)<\infty\) almost surely, and
\[
 \sup_n\int_{|\xi|>L}|\xi|^2\mu_{t_n}(d\xi)
 \le R_\mu L^{2-p_0}\longrightarrow0
 \quad(t_n\to t).
\]
Thus weak path continuity upgrades to \(\mathcal W_2\)-continuity.  Compact containment, the metric-space Kolmogorov criterion, and Arzel\`a--Ascoli give tightness of the measure coordinate.  Hence \(\mathfrak W\) is tight in \(\mathcal P(\Omega^{\mathrm c})\).

Let \(P_n\in\mathfrak W\) and \(P_n\Rightarrow P\).  Use the weakly lower semicontinuous cost
\[
 \mathsf C_2(\nu,\nu')
 :=\inf_{\pi\in\Pi(\nu,\nu')}
 \int_{\mathbb R^{d_x}\times\mathbb R^{d_x}}|\xi-\xi'|^2\pi(d\xi,d\xi')
 \in[0,\infty].
\]
Set the nonnegative lower semicontinuous functionals
\[
 G(X,\mu):=\|X\|_T^{p_0}+\sup_{r\le T}M_{p_0}(\mu_r),\qquad
 F_{s,t}(X,\mu):=|X_t-X_s|^{p_{\rm tm}}
 +\mathsf C_2(\mu_t,\mu_s)^{p_{\rm tm}/2}.
\]
Portmanteau gives
\[
 \mathcal M_{p_0}(P)=\mathbb E^P G
 \le\liminf_{n\to\infty}\mathbb E^{P_n}G
 \le K_1.
\]
Thus \(\mu_r\in\mathcal P_2\) almost surely for every \(r\), and \(\mathsf C_2(\mu_t,\mu_s)=\mathcal W_2^2(\mu_t,\mu_s)\).  A second application gives
\[
 \mathbb E^P\!\left[
 |X_t-X_s|^{p_{\rm tm}}
 +\mathcal W_2^{p_{\rm tm}}(\mu_t,\mu_s)\right]
 =\mathbb E^P F_{s,t}
 \le\liminf_{n\to\infty}\mathbb E^{P_n}F_{s,t}
 \le K_2|t-s|^{1+\alpha_{\rm tm}}.
\]
The fixed marginal passes to the limit, and \cref{lem:compatible_law_closed} gives \(P\in\mathfrak K\).  Hence \(P\in\mathfrak W\); the set is closed, and Prokhorov completes the proof.
\end{proof}

\subsection{Proof of Lemma \ref{lem:response_estimates}}
\label{app:proof_response_estimates}

\begin{proof}
Fix \(P\in\mathfrak W\) and set \(Q=\mathcal R(P)\). The closed-loop growth bound, BDG, H\"older, Gronwall, \eqref{eq:derived_conditional_moment_bound}, and (H7)(iii) give
\begin{gather*}
 \mathcal M_{p_0}(Q)
 \le A_{p_0}^X(T)+(1+c_{p_0}(T))C_{\rm cond}=:A_{p_0}(T),\\
 \mathbb E^Q|X_t-X_s|^{p_{\rm tm}}
 \le B_X(T)|t-s|^{p_{\rm tm}/2}
  =B_X(T)|t-s|^{1+\alpha_{\rm tm}},\\
 \mathbb E^Q\mathcal W_2^{p_{\rm tm}}(\mu_t,\mu_s)
 \le C_{\rm tm}|t-s|^{1+\alpha_{\rm tm}}.
\end{gather*}
These are the claimed bounds, with \(B_\mu(T)=C_{\rm tm}\).
\end{proof}

\subsection{Proof of Lemma \ref{lem:T_continuity}}
\label{app:proof_T_continuity}

\begin{proof}
Let \(\{P_n\}_{n\ge1}\subset\mathfrak W\) satisfy \(P_n\xrightarrow{w}P\). Skorokhod representation \cite[Theorem 6.7]{billingsley1999convergence}, the uniform \(p_0\)-moment bound, and \(p_0>2\) give
\begin{equation}
 \mathbf d_T(P_n,P)\longrightarrow0.
 \label{eq:joint_distance_from_weak}
\end{equation}

On one basis carrying \((X_0,V)\), set \(\eta^{P,0}=\delta_0\) and
\[
 X^{P,k}:=X^{P,\eta^{P,k}},\qquad
 \eta^{P,k+1}:=m^{P,\eta^{P,k}}.
\]
For fixed \(k\), write
\[
 D_k^n(t):=\mathbb E\|X^{P_n,k}-X^{P,k}\|_t^2,\qquad
 B_k^n(t):=\mathbb E\!\left[\sup_{s\le t}
 \mathcal W_2^2(\eta_s^{P_n,k},\eta_s^{P,k})\right].
\]
The selector, state, and filter estimates give
\[
 D_k^n(t)\le C\int_0^t\{D_k^n(s)+B_k^n(s)\}ds
 +C\mathbf d_T^2(P_n,P),\qquad
 B_{k+1}^n(t)\le L_{\rm fil}D_k^n(t).
\]
Since \(B_0^n=0\), induction and Gronwall give, for fixed \(k\),
\begin{equation}
 D_k^n(T)+B_{k+1}^n(T)\longrightarrow0.
 \label{eq:finite_stage_continuity}
\end{equation}

Let \(Q^{P,k}:=\mathcal L(X^{P,k},V,\eta^{P,k+1})\). For fixed \(k\), \eqref{eq:finite_stage_continuity} gives \(\mathbf d_T(Q^{P_n,k},Q^{P,k})\to0\), while the factorial estimate established in Appendix~\ref{app:proof_joint_response_wellposed} is uniform over \(P\in\mathfrak W\):
\begin{equation}
 \sup_{P\in\mathfrak W}
 \mathbf d_T(Q^{P,k},\mathcal R(P))\longrightarrow0
 \quad\text{as }k\to\infty.
 \label{eq:uniform_recursive_tail}
\end{equation}
The triangle inequality, first with \(n\to\infty\) and then \(k\to\infty\), gives \(\mathbf d_T(\mathcal R(P_n),\mathcal R(P))\to0\).
\end{proof}

\subsection{Proof of Lemma \ref{lem:fixed_point_optimal_response_identification}}
\label{app:proof_fixed_point_optimal_response_identification}

\begin{proof}
Let \((\widetilde X,\widetilde V,\widetilde\mu)\) construct \(\mathcal R(P)\). Since \(\mathcal R(P)=P\),
\[
 \mathcal L(\widetilde X_0,\widetilde V,\widetilde\mu)
 =P\circ(X_0,V,\mu)^{-1}=\Gamma_P .
\]
Condition \emph{(EP)} gives the same environment law under \(\mathbb P^{P,*}\):
\[
 \mathbb P^{P,*}\circ
 (X_0^P,V^{P,\phi^P},\eta)^{-1}=\Gamma_P .
\]
Uniqueness in law from \cref{lem:closed_loop_input_stability}, applied to this environment and \(\phi^P\), yields
\[
 \mathbb P^{P,*}\circ
 (X^P,V^{P,\phi^P},\eta)^{-1}
 =\mathcal L(\widetilde X,\widetilde V,\widetilde\mu)=P .
\]

For \(\phi\in\Phi_{ad}\), let \(\widetilde X^\phi\) be the corresponding strong solution on the fixed-point environment.  Condition \emph{(EP)} and the same uniqueness-in-law result give
\begin{equation}
 \mathbb P^{P,\phi}\circ(X^P,V^{P,\phi},\eta)^{-1}
 =\mathcal L(\widetilde X^\phi,\widetilde V,\widetilde\mu).
 \label{eq:deviation_law_transfer}
\end{equation}
Thus the deviation and equilibrium costs are \(J_P(\phi)\) and \(J_P(\phi^P)\).  Theorem \ref{thm:fixed_P_existence} yields
\[
 J_P(\phi^P)\le J_P(\phi),\qquad \phi\in\Phi_{ad},
\]
on the fixed-point realization and, by the same law argument, on every strong compatible realization with joint law \(P\).
\end{proof}

\subsection{Proof of Theorem \ref{thm:weak_equilibrium}}
\label{app:proof_weak_equilibrium}

\begin{proof}
By \cref{lem:W_compact,lem:T_continuity} and invariance, \(\mathcal R\) is a continuous self-map of the nonempty compact convex set \(\mathfrak W\) in the locally convex space of finite signed measures with the weak topology.  The Schauder--Tychonoff theorem \cite[Corollary 17.56]{aliprantis2006infinite} gives \(P^*\in\mathfrak W\) with \(\mathcal R(P^*)=P^*\).

The response equations make this fixed point a compatible self-consistent weak solution, and \cref{lem:fixed_point_optimal_response_identification} makes \(\phi^{P^*}\) optimal.  Hence \(P^*\) is a weak WPR-MFG equilibrium.
\end{proof}

\subsection{Proof of Theorem \ref{thm:strong_equilibrium}}
\label{app:proof_strong_equilibrium}

\begin{proof}
For the fixed point \(P^*\), \cref{lem:joint_response_wellposed} gives a strong compatible response driven by \((X_0,V)\), with law \(P^*\), and \cref{lem:fixed_point_optimal_response_identification} makes its control optimal.  For two jointly compatible responses driven by the same input and using \(\phi^{P^*}\), let \(D(t):=\mathbb E\|X^1-X^2\|_t^2\), the coefficient bounds give
\[
 D(t)\le C\int_0^t\left\{D(s)+
 \mathbb E\!\left[\sup_{r\le s}
 \mathcal W_2^2(\mu_r^1,\mu_r^2)\right]\right\}ds.
\]
Assumption (H7)(ii) bounds the second term by \(L_{\rm fil}D(s)\), so Gronwall gives \(X^1=X^2\), then \(\mu^1=\mu^2\).  On any compatible basis with the prescribed input law, the same strong realization follows from this pathwise uniqueness and the compatible Yamada--Watanabe theorem \cite[Theorem 1.5]{kurtz2014weak}.  Its hypotheses apply because the state equation, compatibility, and conditional-law identity have countable Borel formulations, while \cite[Lemmas 2.10--2.11]{kurtz2014weak} supplies the required coupling.
\end{proof}
\section{Proofs for the LQ Example}
\label{app:lq_new}

\subsection{Proof of Proposition \ref{prop:lq_frozen_response}}
\label{app:proof_lq_frozen_response}

For either LQ control space, write \(\chi_t^u=\mathbb E[x_t^u\mid\mathcal F_t^Y]\) and \(\zeta_t^u=x_t^u-\chi_t^u\).  Since \(\zeta^u\) solves the second equation in \eqref{eq:lq_projection}, its cost contribution is independent of \(u\). After conditioning, \(u\mapsto\chi^u\) is affine and continuous, while the quadratic part has Hessian
\(\left(\begin{smallmatrix}\mathsf R_u&-\mathsf q_{ux}\\
-\mathsf q_{ux}&\mathsf Q_x\end{smallmatrix}\right)\), positive definite by \eqref{eq:lq_strict_convexity}.  Young's inequality and the direct method give a unique minimizer on either closed control subspace.  The projected first variation yields the adjoint and conditional stationarity relations; imposing \(\chi^u=m\) and \(\overline u=u\) reduces them to \eqref{eq:lq_common_optimality}.  Strict convexity makes this condition sufficient.

\subsection{Proof of Theorem \ref{thm:lq_pr_wpr_new}}
\label{app:proof_lq_pr_wpr_new}

\paragraph{PR system}
Under PR, \eqref{eq:lq_common_optimality} gives \(\mathsf R_{\rm eff}u_t=\widehat\Theta_t-\varpi_t\).  Substituting
\[
 \varpi_t=\mathfrak r_t(m_t-c)+\mathfrak s_t\widehat\Theta_t
\]
and matching the coefficients of \(m_t-c\) and \(\widehat\Theta_t\) gives \eqref{eq:lq_pr_riccati_new}, with martingale coefficient \(\mathfrak s_tK_t^{\rm f}\sigma_y\).  The scalar Riccati equation is globally solvable backward from \(\mathsf Q_T\), and the second equation is then linear. Together with the well-posed Kalman--Bucy filter, this gives \eqref{eq:lq_pr_control_new} and a linear closed-loop system.  To show that the ansatz covers every PR solution, define
\[
 \varepsilon_t:=\varpi_t-\mathfrak r_t(m_t-c)
                 -\mathfrak s_t\widehat\Theta_t.
\]
The two equations in \eqref{eq:lq_pr_riccati_new} give
\[
 d\varepsilon_t
 =\left(a+\frac{\mathsf q_{ux}+\mathfrak r_t}
                   {\mathsf R_{\rm eff}}\right)\varepsilon_tdt
   +d\widetilde{\mathcal N}_t,
 \qquad \varepsilon_T=0,
\]
where \(d\widetilde{\mathcal N}_t =d\mathcal N_t-\mathfrak s_tK_t^{\rm f}\sigma_y d\widehat B_t^Y\). The deterministic integrating factor makes \(\varepsilon\) a square-integrable martingale with zero terminal value.  Thus \(\varepsilon=0\), and the PR system has a unique strong solution.

\paragraph{WPR equation}
Solving the adjoint equation backward and conditioning on \(\sigma(y_t)\) gives \eqref{eq:lq_wpr_exact_new}.  First work on the full closed space \(\mathbb H_{\rm W}:=\mathcal U_{\rm WPR}^{\rm LQ}\), before using Gaussianity. Its orthogonal projection has the disintegrated version
\[
 (\mathsf P_{\rm W}z)_t=\mathbb E[z_t\mid y_t]
 \quad\text{for }dt\text{-a.e. }t,
\]
and is an \(L^2\)-contraction.  For \(v\in\mathbb H_{\rm W}\), set \(\mathfrak a_t(v)=\int_0^t v_sds\), let \(\varpi^v\) solve the linear adjoint with terminal value \(\mathsf Q_T\mathfrak a_T(v)\), and define \((\mathfrak Bv)_t:=\mathsf R_{\rm eff}v_t+ (\mathsf P_{\rm W}\varpi^v)_t\).  Conditional Jensen shows that \(\mathfrak B\) is bounded in the norm \(\|v\|_a^2=\mathbb E\int_0^T e^{-at}|v_t|^2dt\).  With \(\mathfrak b(v,w):=\langle\mathfrak Bv,w\rangle_a\), the product rule for \(e^{-at}\mathfrak a_t(v)\varpi_t^v\) gives
\begin{equation}
\begin{aligned}
 \mathfrak b(v,v)
 ={}&\mathsf R_{\rm eff}\|v\|_a^2
 +e^{-aT}\left(\mathsf Q_T+\frac{\mathsf q_{ux}}2\right)
   \mathbb E[\mathfrak a_T(v)^2]\\
 &+\frac{a\mathsf q_{ux}}2
   \mathbb E\int_0^T e^{-at}\mathfrak a_t(v)^2dt
 \ge \mathsf R_{\rm eff}\|v\|_a^2.
\end{aligned}
\label{eq:lq_wpr_coercivity_new}
\end{equation}

Let \(\mathfrak l\) be the weighted linear functional from the right-hand side of \eqref{eq:lq_wpr_exact_new}.  Galerkin solutions on nested dense finite-dimensional spaces exist and are uniformly bounded by \eqref{eq:lq_wpr_coercivity_new}.  Weak compactness and continuity give
\[
 \mathfrak b(v,w)=\mathfrak l(w),\qquad w\in\mathbb H_{\rm W}.
\]
The same estimate gives uniqueness.  Testing the residual in \eqref{eq:lq_wpr_exact_new} against itself makes it vanish, so the equation has a unique solution in the full WPR class.

\paragraph{Gaussian reduction and closure}
For \(t>0\), \(\mathsf V_t^y>0\), so \(\widetilde y_t=(y_t-\overline y_t)/\sqrt{\mathsf V_t^y}\) is standard Gaussian.  The normalized Hermite basis yields \(\mathbb H_{\rm W}=\bigoplus_{j\ge0}\mathbb H_j\), where \(\mathbb H_j\) is generated by \(\zeta_j(t)e_j(\widetilde y_t)\).  Joint Gaussianity gives
\[
 \mathbb E[e_j(\widetilde y_s)\mid y_t]
 =\rho_{st}^{\,j}e_j(\widetilde y_t),
 \qquad \rho_{st}:=\operatorname{Corr}(y_s,y_t).
\]
For \(v_t=\zeta_j(t)e_j(\widetilde y_t)\), conditional Fubini gives
\[
\begin{aligned}
 (\mathfrak Bv)_t=\Bigg[&\mathsf R_{\rm eff}\zeta_j(t)
 +e^{-a(T-t)}\mathsf Q_T\int_0^T\zeta_j(s)\rho_{st}^{\,j}ds\\
 &+\mathsf q_{ux}\int_t^T e^{-a(s-t)}
       \zeta_j(s)\rho_{st}^{\,j}ds\Bigg]e_j(\widetilde y_t).
\end{aligned}
\]
Hence \(\mathfrak B\) preserves every \(\mathbb H_j\).  Its right-hand side has only orders zero and one, so uniqueness eliminates the higher modes.  This gives \eqref{eq:lq_wpr_affine_new} and \eqref{eq:lq_wpr_mean_new}--\eqref{eq:lq_wpr_gain_new}; conversely, those equations reconstruct \eqref{eq:lq_wpr_exact_new} and uniquely determine both coefficients.

The coefficient space makes the WPR control square-integrable, so \eqref{eq:lq_wpr_closed_new} has a unique strong solution.  For \(I\in\{\mathrm{PR},\mathrm{WPR}\}\), conditional expectation in the state equation gives
\[
 d\bigl(\mathbb E[x_t^{I}\mid\mathcal F_t^Y]-m_t^{I}\bigr)
 =-a\bigl(\mathbb E[x_t^{I}\mid\mathcal F_t^Y]-m_t^{I}\bigr)dt
\]
with zero initial value.  Thus both consistency conditions hold, and \cref{prop:lq_frozen_response} identifies the unique best responses.  Uniqueness of the two optimality systems proves the claimed equilibrium uniqueness.

\bibliographystyle{siamplain}
\bibliography{references}
\end{document}